\documentclass[12pt,reqno]{amsart}

\usepackage{etex}
\usepackage{geometry}             
\usepackage[dvipsnames]{xcolor}   		             		
\usepackage{graphicx}				
\usepackage{amssymb,mathrsfs}
\usepackage{amsthm,amsmath,stmaryrd}
\usepackage{tikz}
\usepackage{tikz-cd}
\usepackage{accents,upgreek,enumerate}
\usepackage{bm}
\usepackage{mathtools}
\usepackage[all]{xy}
\usepackage{caption}
\usepackage[normalem]{ulem}

\usepackage{enumitem}
\usepackage{pdflscape} 
\usepackage{multirow}
\usepackage{fancyhdr} 
\usepackage{makecell}
\usepackage{boldline}

\usepackage{subcaption}

\fancypagestyle{lscapedplain}{%
  \fancyhf{}
  \fancyfoot{%
    \tikz[remember picture,overlay]
      \node[outer sep=1cm,above,rotate=90] at (current page.east) {\thepage};}
}

\tikzset{
  commutative diagrams/.cd, 
  arrow style=tikz, 
  diagrams={>=stealth}
}

  \makeatletter
\def\@tocline#1#2#3#4#5#6#7{\relax
  \ifnum #1>\c@tocdepth 
  \else
    \par \addpenalty\@secpenalty\addvspace{#2}%
    \begingroup \hyphenpenalty\@M
    \@ifempty{#4}{%
      \@tempdima\csname r@tocindent\number#1\endcsname\relax
    }{%
      \@tempdima#4\relax
    }%
    \parindent\z@ \leftskip#3\relax \advance\leftskip\@tempdima\relax
    \rightskip\@pnumwidth plus4em \parfillskip-\@pnumwidth
    #5\leavevmode\hskip-\@tempdima
      \ifcase #1
       \or\or \hskip 1em \or \hskip 2em \else \hskip 3em \fi%
      #6\nobreak\relax
    \dotfill\hbox to\@pnumwidth{\@tocpagenum{#7}}\par
    \nobreak
    \endgroup
  \fi}
\makeatother

\usetikzlibrary{calc}
\usetikzlibrary{fadings}
\usetikzlibrary{decorations.pathmorphing}
\usetikzlibrary{decorations.pathreplacing}

\newcounter{marginnote}
\DeclareMathAlphabet{\mathpzc}{OT1}{pzc}{m}{it}

\usepackage[backref=page]{hyperref}
\hypersetup{
  colorlinks   = true,          
  urlcolor     = blue,          
  linkcolor    = violet,          
  citecolor   = blue             
}

\newtheorem{theorem}{Theorem}[section]

\newtheorem{corollary}[theorem]{Corollary}
\newtheorem{lemma}[theorem]{Lemma}
\newtheorem{proposition}[theorem]{Proposition}

\newtheorem{quasi-theorem}[theorem]{Quasi-Theorem}

\theoremstyle{definition}
\newtheorem{definition}[theorem]{Definition}

\newtheorem{remark}[theorem]{Remark}

\newtheorem{example}[theorem]{Example}
\newtheorem{blank remark}[theorem]{}

\newtheorem{not1}[theorem]{Notation}

\newcommand{\NN} {{\mathbb N}}		
\newcommand{\PP}{\mathbb{P}}         
		
\newcommand{\RR} {{\mathbb R}}		
\newcommand{\ZZ} {{\mathbb Z}}

\def\setminus{\smallsetminus}

\DeclareMathOperator{\Aut}{Aut}

\DeclareMathOperator{\val}{val}

\DeclareMathOperator{\mult}{mult}

\DeclareMathOperator{\Br}{Br}

\newcommand{\cal}{\mathcal}

\def\cM{{\cal M}}

\newcommand{\Mbar}{\overline{\cM}}

\def\trop{\mathrm{trop}}

\newcommand{\lightening}{{\mathsf{log}}}
\newcommand{\leak}{\mathsf{DR}_g^{\trop}(\mathbf x)}

\title[Intersection numbers of logDR]{$k$-leaky double Hurwitz descendants}

\author{Renzo Cavalieri} 
\address{Colorado State University, Department of Mathematics, Weber Building, Fort Collins, CO 80523-1874, USA}
\email{renzo@math.colostate.edu}

\author {Hannah Markwig}
\address {Universit\"at T\"ubingen, Fachbereich Mathematik, Auf der Morgenstelle 10, 72076 T\"ubingen, Germany }
\email {hannah@math.uni-tuebingen.de}

\author{Johannes Schmitt}
\address{Department Mathematik, R\"amistrasse 101, CH-8092 Z\"urich, Switzerland}
\email{johannes.schmitt@math.ethz.ch}

\DeclareMathOperator{\logDR}{logDR}
\DeclareMathOperator{\DR}{DR}

\DeclareMathOperator{\CH}{CH}
\DeclareMathOperator{\logCH}{logCH}
\DeclareMathOperator{\br}{br}

\def\cM{{\mathcal{M}}}
\def\Mbar{\overline{\mathcal{M}}}
\def\bfx{\mathbf{x}}
\def\bfe{\mathbf{e}}

\def\PP{\mathbb{P}}

\begin{document}

\begin{abstract}
We define a new class of enumerative invariants called $k$-leaky double Hurwitz descendants, generalizing both descendant integrals of double ramification cycles and the $k$-leaky double Hurwitz numbers introduced in \cite{CMR22}. These numbers are defined as intersection numbers of the logarithmic DR cycle against $\psi$-classes and logarithmic classes coming from piecewise polynomials encoding fixed branch point conditions. We give a tropical graph sum formula for these new invariants, allowing us to show their piecewise polynomiality and a wall-crossing formula in genus zero. We also prove that in genus zero the invariants are always non-negative and give a complete classification of the cases where they vanish. 

\end{abstract}

\maketitle

\tableofcontents
\section{Introduction}

The goal of this paper is to study enumerative invariants of the (logarithmic) double ramification cycles, giving an overview of the existing landscape of results and defining new invariants using techniques from logarithmic and tropical geometry.

\subsection{The double ramification cycle and its intersection numbers}
Inside the moduli space $\cM_{g,n}$ of smooth curves, the \emph{double ramification locus} is a closed substack cut out by an equality of line bundles
\begin{equation} \label{eqn:DRL}
    \mathrm{DRL}_g(\bfx) = \left\{(C, p_1, \ldots, p_n) : (\omega_C^{\mathrm{log}})^{\otimes k} \cong \mathcal{O}_C\left(\sum_{i=1}^n x_i p_i\right)\right\} \subseteq \cM_{g,n}\,,
\end{equation}
where $\mathbf{x}=(x_1, \ldots, x_n) \in \mathbb{Z}^n$ is an integer vector with $|\bfx| = \sum_i x_i = k(2g-2+n)$.
The (virtual) fundamental class of $\mathrm{DRL}_g(\mathbf{x})$ admits a natural extension
\begin{equation} \label{eqn:DR}
    \DR_g(\bfx) \in \CH_{2g-3+n}(\Mbar_{g,n}),
\end{equation}
in the Chow group of dimension $2g-3+n$ cycles on $\Mbar_{g,n}$, called the \emph{double-ramification cycle (class)}.\footnote{Since $k$ can be uniquely reconstructed from $g,n$ and $\bfx$, we omit it from the notation. }
For $k=0$, this class can be defined as the pushforward of the virtual fundamental class of the space $\Mbar^{\sim}_{g,\bfx}$  of stable maps to
{\em rubber} \cite{Graber2005Relative-virtua, Li2001Stable-morphism, Li2002A-degeneration-}. This is a compactification of the space of maps
\begin{equation} \label{eqn:DRcover}
f: (C,p_1, \ldots, p_n) \to \mathbb{P}^1\text{ with }f^*([0]-[\infty]) = \sum_{i=1}^n x_i p_i,
\end{equation}
with ramification profile over $0,\infty$ given by the positive and negative parts of $\bfx$ (and taken modulo the 
$\mathbb{C}^*$-action on $\PP^1$). For the history and general properties of the double ramification cycle we refer the reader to \cite{JPPZ}.

An established approach to extract intersection numbers from the cycle $\DR_g(\bfx)$ is to calculate its \emph{descendant invariants}
\begin{equation} \label{eqn:DRdescendants}
    \int_{\Mbar_{g,n}} \DR_g(\bfx) \cdot \psi_1^{e_1} \cdots \psi_n^{e_n},
\end{equation}
for a vector $\bfe = (e_1, \ldots, e_n) \in (\mathbb{Z}_{\geq 0})^n$ of exponents with $|\bfe| = 2g-3+n$. An explicit formula for these numbers was calculated in \cite{BSSZ} for $k=0$ and in \cite{CSS} for  arbitrary $k$,  but with $\bfe$ restricted to vectors with only a single nonzero entry. A separate direction of study is the \emph{double ramification hierarchy}, which studies intersection numbers of $\DR_g(\bfx)$ against Hodge classes, a (partial) cohomological field theory and $\psi$-classes as above (see \cite{Buryak2015Double-ramifica, Buryak2021Quadratic}).

There are natural enumerative questions concerning the double ramification geometry which have no known expression via intersection numbers of $\DR_g(\bfx)$ on $\Mbar_{g,n}$. A prominent example are the \emph{double Hurwitz numbers} $\mathrm{H}_g(\bfx)$ \cite{Goulden2005Towards-the-geo}. These count covers $f$ as in \eqref{eqn:DRcover} with $b=2g-3+n$ simple branch points at fixed positions in $\mathbb{P}^1$. The double Hurwitz numbers \emph{can} be defined by intersection theory: the space $\Mbar^{\sim}_{g,\bfx}$ of rubber maps has a natural branch morphism
\begin{equation} \label{eqn:brmorphism}
    \mathrm{br}: \Mbar^{\sim}_{g,\bfx} \to [\mathrm{LM}(b)/S_b]
\end{equation}
to a (stack quotient of a) \emph{Losev-Manin space}, which remembers the position of the simple branch points and allows them to coincide away from $0, \infty$ \cite{cm:geomperspdhn}. Then we have
\begin{equation} \label{eqn:doubleHurwitzviabr}
    \mathrm{H}_g(\bfx) = \int_{[\Mbar^{\sim}_{g,\bfx}]^\mathrm{vir}} \mathrm{br}^*[\mathrm{pt}].
\end{equation}
Since the branch morphism \eqref{eqn:brmorphism} does not factor through the forgetful morphism $F : \Mbar^{\sim}_{g,\bfx} \to \Mbar_{g,n}$, there is no obvious way to express the intersection number \eqref{eqn:doubleHurwitzviabr} as a product of $\DR_g(\bfx) = F_* [\Mbar^{\sim}_{g,\bfx}]^\mathrm{vir}$ with a class on $\Mbar_{g,n}$.

\subsection{Logarithmic double ramification cycles and double Hurwitz numbers}
The failure of $\Mbar_{g,n}$ in supporting a class that restricts to $\mathrm{br}^* [\mathrm{pt}]$ is part of a broader picture that emerged in recent years, and which can be summarized as follows:
\begin{align*}
    \textit{$\Mbar_{g,n}$ is not the right ambient space for the double ramification cycle.}
\end{align*}
The first hint of this appeared in a construction of $\DR_g(\bfx)$ 
\cite{Holmes2017Extending-the-d,MW17}, which proceeds by constructing a cycle $\widehat{\DR}_g(\bfx)$ on a \emph{log blowup}\footnote{See e.g. \cite[Section 2]{Barrott2019Logarithmic-Cho} for a definition of log blowups. In the context of our paper, the reader can think of lob blowups of $\Mbar_{g,n}$ as iterated blowups at smooth boundary strata and their strict transforms.} $\widehat{\cal M}^\bfx_{g,n}$ of $\Mbar_{g,n}$ and then obtains $\DR_g(\bfx)$ as the pushforward of $\widehat{\DR}_g(\bfx)$ under the map $\widehat{\cal M}^\bfx_{g,n} \to \Mbar_{g,n}$. This construction works for arbitrary $k \in \mathbb{Z}$  (compared to the one via stable maps to rubber when $k=0$) and gives a natural lift of $\DR_g(\bfx)$ to a {log blowup} of $\Mbar_{g,n}$. While the blowup $\widehat{\cal M}^\bfx_{g,n}$ is not unique, 
the resulting cycle
stabilizes on sufficiently fine blowups: given another (sufficiently fine) choice of blowup $\widetilde{\cal M}^\bfx_{g,n}$, the cycle $\widetilde{\DR}_g(\bfx)$ constructed there pulls back to the same cycle as $\widehat{\DR}_g(\bfx)$ on any blowup dominating $\widehat{\cal M}^\bfx_{g,n}$ and $\widetilde{\cal M}^\bfx_{g,n}$. Thus one obtains a well-defined element in the \emph{logarithmic Chow ring}
\begin{equation} \label{eqn:logCH}
    \logCH^*(\Mbar_{g,n}) = \varinjlim_{\widehat{\mathcal{M}} \to \Mbar_{g,n}} \CH^*(\widehat{\mathcal{M}}),
\end{equation}
defined as the direct limit of Chow rings of smooth log-blowups $\widehat{\mathcal{M}} \to \Mbar_{g,n}$, with 
maps given by pullback. The constructed lift
\begin{equation} \label{eqn:logDR}
    \logDR_g(\bfx) = [\widehat{\cal M}^\bfx_{g,n}, \widehat{\DR}_g(\bfx)] \in \logCH^*(\Mbar_{g,n})
\end{equation}
is called the \emph{logarithmic double-ramification cycle}.

 One immediate advantage of $\logDR_g(\bfx)$ compared to $\DR_g(\bfx)$ is that we can  recover the double Hurwitz numbers from $\logDR_g(\bfx)$. Indeed, as shown in \cite[Theorem A]{CMR22} there exists a cycle $\mathrm{Br}_g(\bfx) \in \logCH^{2g-3+n}(\Mbar_{g.n})$ such that
\begin{equation} \label{eqn:doubleHurwitzlogDR}
    \mathrm{H}_g(\bfx) = \int_{\Mbar_{g,n}} \logDR_g(\bfx) \cdot \Br_g(\bfx).
\end{equation}
To see where the cycle $\Br_g(\bfx)$ comes from, recall from equation \eqref{eqn:doubleHurwitzviabr} that we want to pair the virtual fundamental class of $\Mbar^{\sim}_{g,\bfx}$ with $\br^*[\mathrm{pt}]$. For a suitable choice of blowup $\widehat{\mathcal{M}}^\bfx_{g,n} \to \Mbar_{g,n}$ one can ensure that the map $\Mbar^{\sim}_{g,\bfx} \to \Mbar_{g,n}$ factors via an embedding $\iota$ into $\widehat{\mathcal{M}}^\bfx_{g,n}$ as illustrated on the left-hand side of Figure \ref{fig:rubfactorization} (see Proposition \ref{pro:tmap}). Moreover, the pushforward of the virtual class of $\Mbar^{\sim}_{g,\bfx}$ under $\iota$ gives the lifted double ramification cycle $\widehat{\DR}_{g}(\bfx)$ (see e.g. \cite[Proposition 7.1]{Holmes2017Extending-the-d}).

While the map $\br$ does not immediately factor through $\widehat{\mathcal{M}}^\bfx_{g,n}$, we can nevertheless get a proxy for the cycle $\br^*[\mathrm{pt}]$ using an auxiliary space: the stack $\mathsf{Ex}$  of \emph{expansions} parameterizes chains of rational curves with marked points $0,\infty$ at the opposite ends of the chain (see Figure \ref{fig:ex},  and \cite[Section 2.8]{CMR22} for a discussion).
\begin{figure}[tb]
    \centering
    \begin{tikzcd}
     &[-2em] \mathrm{br}^*{[\mathrm{pt}]}  & {[\mathrm{pt}]} \arrow[l, mapsto, "\mathrm{br}^*", swap] & {[T_b]} \arrow[l, mapsto, "F_b^*", swap] \\[-1.5em]
    {[\Mbar^{\sim}_{g,\bfx}]^\mathrm{vir}} \arrow[d, mapsto, "\iota_*",swap] & \Mbar^{\sim}_{g,\bfx} \arrow[d, "\iota", swap] \arrow[r,"\mathrm{br}"] & {[\mathrm{LM}(b)/S_b]} \arrow[r, "F_b"] & \mathsf{Ex}\\
    \widehat{\DR}_{g}(\bfx) \arrow[d, mapsto, "q_*", swap] & \widehat{\mathcal{M}}^\bfx_{g,n} \arrow[d, "q", swap] \arrow[rru, bend right=10,"\mathsf{t}"] & &\\
    \DR_g(\bfx) & \Mbar_{g,n}
    \end{tikzcd}
    \caption{Factoring the branch morphism to the stack $\mathsf{Ex}$ of expansions through the log blowup $\widehat{\mathcal{M}}^\bfx_{g,n}$}
    \label{fig:rubfactorization}
\end{figure}

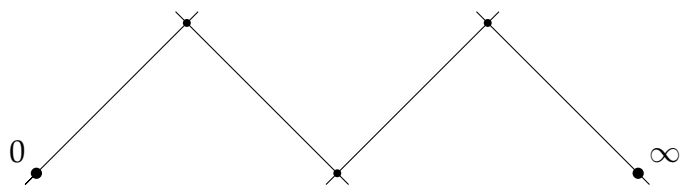
\begin{figure}[tb]
    \centering
\begin{tikzpicture}
  \def\linelength{2cm}
  \def\extension{0.15cm}

  \coordinate (A) at (0,0);
  \coordinate (B) at (\linelength, \linelength);
  \coordinate (C) at (2*\linelength, 0);
  \coordinate (D) at (3*\linelength, \linelength);
  \coordinate (E) at (4*\linelength, 0);

  \draw (A)--(-\extension,-\extension) -- ($(B)+(\extension, \extension)$);
  \draw ($(B)+(-\extension, \extension)$) -- ($(C)+(\extension, -\extension)$);
  \draw ($(C) + (-\extension,-\extension)$) -- ($(D)+(\extension, \extension)$);
  \draw ($(D) + (-\extension,\extension)$) -- ($(E)+(\extension, -\extension)$);

  \fill (B) circle (1.5pt);
  \fill (C) circle (1.5pt);
  \fill (D) circle (1.5pt);
  
  \node[fill,circle,inner sep=1.5pt] (left) at (A) {};
  \node[fill,circle,inner sep=1.5pt] (right) at (E) {};
  \node[above left] at (left) {0};
  \node[above right] at (right) {$\infty$};
\end{tikzpicture}
    \caption{A point in the stack of expansions $\mathsf{Ex}$.}
    \label{fig:ex}
\end{figure}
For any $c \geq 0$ the stack $\mathsf{Ex}$ has a closed codimension $c$ stratum $T_c$ parameterizing chains of length at least $c$. Then for the map $F_b : [\mathrm{LM}(b)/S_b] \to \mathsf{Ex}$ forgetting the $b=2g-3+n$ unordered points, it is easy to see that the class $[T_b]$ of the codimension $b$ stratum pulls back to the class of a point in the Losev-Manin space. On the other hand, by Proposition \ref{pro:tmap} the composition $F_b \circ \br$ factors through a map $\mathsf{t}$ as illustrated in Figure \ref{fig:rubfactorization}. Then defining $\Br_g(\bfx) = \mathsf{t}^* [T_b]$ we conclude
\begin{align*} 
    \mathrm{H}_g(\bfx) &= \int_{[\Mbar^{\sim}_{g,\bfx}]^\mathrm{vir}} \mathrm{br}^*[\mathrm{pt}] = \int_{[\Mbar^{\sim}_{g,\bfx}]^\mathrm{vir}} \mathrm{br}^* F_b^* [T_b] = \int_{\widehat{\mathcal{M}}^\bfx_{g,n}} (\iota_*[\Mbar^{\sim}_{g,\bfx}]^\mathrm{vir}) \cdot \mathsf{t}^* [T_b]\\
    &=\int_{\widehat{\mathcal{M}}^\bfx_{g,n}} \widehat{\DR}_{g}(\bfx) \cdot \Br_g(\bfx)=\int_{\Mbar_{g,n}} \logDR_g(\bfx) \cdot \Br_g(\bfx)\,.
\end{align*}

\subsection{Generalized double Hurwitz numbers}
In the previous section, we saw a formula for the double Hurwitz numbers in terms of intersection numbers of $\logDR_g(\bfx)$. Now, turning the story around one may \emph{define} a generalization of these Hurwitz numbers by pairing $\logDR_g(\bfx)$ with more general cycle classes. 

The first step in this direction was already discussed in \cite{CMR22}: while the double Hurwitz numbers require the entries of the ramification profile $\bfx$ to sum to zero (corresponding to setting $k=0$), the formula \eqref{eqn:doubleHurwitzlogDR} can be extended to arbitrary values of $k \in \mathbb{Z}$. More precisely, by Proposition \ref{pro:tmap} we  have the commutative diagram depicted in Figure \ref{fig:maptoex}.
\begin{figure}[tb]
    \centering
  \begin{tikzcd}
    {[\Mbar^{\sim}_{g,\bfx}]^\mathrm{vir}} \arrow[d, mapsto, "\iota_*",swap] &[-2em] \Mbar^{\sim}_{g,\bfx} \arrow[d, "\iota", swap] \arrow[r,"\mathsf{b}"] &  \mathsf{Ex}\\
    \widehat{\DR}_{g}(\bfx) \arrow[d, mapsto, "q_*", swap] & \widehat{\mathcal{M}}^\bfx_{g,n} \arrow[d, "q", swap] \arrow[ru, bend right=10,"\mathsf{t}"] & \\
    \DR_g(\bfx) & \Mbar_{g,n}
    \end{tikzcd}
    \caption{Maps to the stack of expansions.}
    \label{fig:maptoex}
\end{figure}
  
Thus the definition $\Br_g(\bfx) = \mathsf{t}^* [T_b]$ still makes sense, and the paper \cite{CMR22} defines the \emph{$k$-leaky double Hurwitz numbers}\footnote{Again, since $k$ is determined by $g$ and $\bfx$, we will in general omit the parameter $k$ in the notation for generalized double Hurwitz numbers.}
\begin{equation}
    \mathrm{H}_g(\bfx) =  \mathrm{H}_g(\bfx, k) = \int_{\Mbar_{g,n}} \logDR_g(\bfx) \cdot \Br_g(\bfx)\,.
\end{equation}

While the enumerative meaning of the $k$-leaky Hurwitz numbers is less clear, they share many properties of classical double-Hurwitz numbers (\cite[Theorem B and Section 5]{CMR22}):
\begin{itemize}
    \item they are nonnegative  rational numbers, piecewise polynomial in $\bfx$ and given by polynomials of degree $4g-3+n$ on the chambers of polynomiality,
    \item there is a tropical graph sum formula calculating them as a weighted count of tropical covers of the real line ,
    \item they arise as matrix elements for powers of the {\it k-leaky cut and join operator} on the Fock space.
\end{itemize}

\subsection{Results}
The main object of study of the present paper generalizes $k$-leaky Hurwitz numbers by introducing descendant insertions.

\begin{definition}\label{def-hgxe}
    Given a vector $\bfe \in \mathbb{Z}_{\geq 0}^n$ with $0 \leq |\bfe| \leq 2g-3+n$ let $c = 2g-3+n-|e|$ and consider the \emph{branch class}
    \begin{equation}
        \Br_g^c(\bfx) = \mathsf{t}^* [T_c] \in \logCH^c(\Mbar_{g,n})\,.
    \end{equation}
    Then we define the \emph{$k$-leaky double Hurwitz descendants}
    \begin{equation} \label{eqn:Hxe_def}
        \mathrm{H}_g(\bfx, \bfe) = \int_{\Mbar_{g,n}} \logDR_g(\bfx) \cdot \psi_1^{e_1} \cdots \psi_n^{e_n} \cdot \Br_g^c(\bfx)\,.
    \end{equation}
\end{definition}
For $\bfe = 0$ we recover the $k$-leaky double Hurwitz numbers, whereas for $|e|=2g-3+n$ by the projection formula we obtain the descendant invariants \eqref{eqn:DRdescendants} of the double ramification cycle. Thus these numbers form a natural interpolation between two previously studied enumerative invariants.

Our first result establishes a correspondence theorem between $k$-leaky double Hurwitz descendants numbers and certain tropical counts.

\begin{theorem} \label{thm:tropical_formula}
The $k$-leaky double Hurwitz descendants  equal the count of tropical $k$-leaky covers satisfying Psi-conditions (Definition \ref{def-trophgxe}):

$$\mathrm{H}^{\trop}_g(\bfx, \bfe)=\mathrm{H}_g(\bfx, \bfe).$$
\end{theorem}

The graphs that are being counted are leaky covers as in \cite{CMR22}, i.e. piecewise linear maps from tropical curves to the real line satifying the leaky condition \eqref{eq:leakybal}. The valence, and genus at a given vertex depends on the degree of the descendant insertions at  its incident legs \eqref{eq:psicondition}. Each graph is counted with a multiplicity which is a product of various local factors: besides automorphism, and edge factors, there are now multiplicities assigned to each vertex, which consist of $k$-leaky descendants with no appearence of the branch class (Definition \ref{def-multvertex}).  The key ingredient in proving this result is expressing the pushforward to $\Mbar_{g,n}$ of the product $\logDR_g(\bfx) \cdot  \Br_g^c(\bfx)$ as a linear combination of boundary classes that are described by dual graphs decorated with vertex terms given by smaller-dimensional double ramification cycles. 

The correspondence theorem provides a combinatorial approach to the computation of  $k$-leaky double Hurwitz descendants that allows to observe and characterize their structural properties.

\begin{theorem}[Piecewise polynomiality, see Theorem \ref{thm-pp}] \label{thm:main}
  The $k$-leaky double Hurwitz descendant numbers $\mathrm{H}_g(\bfx, \bfe)$ are piecewise polynomial in $\bfx$ of degree $4g-3+n-|\bfe|$. 
    
\end{theorem}

As in the case for double Hurwitz numbers \cite{Cavalieri2011-Wall-crossings}, the wall-crossing formulas are modular, in the sense that they are expressed as sums of products of descendant leaky numbers with smaller discrete invariants. We do not work out general wall-crossing formulas in this paper as the combinatorial complexity seems to outweight the benefits, but we do present the result in genus zero, where the formulas are succint and elegant.

\begin{theorem}[Wall-crossing in genus zero]\label{thm-wc}

Fix a wall $\delta:= \sum_{i\in I} x_i -k\cdot (\sharp I-1)=0 $ and denote by $P_1^\delta$ the polynomial expression for $\mathrm{H}^{\trop}_0(\bfx, \bfe)$ we have on one side of the wall and by $P_2^\delta$ the expression on the other side of the wall.

Define $\bfe_I=(e_i)_{i\in I}$. Let $r=n-2-|\bfe|$, $r_1= \sharp I -1-|\bfe_I|$ and $r_2=\sharp I^c - 1 - |\bfe_{I^c}|$.

Then the wall-crossing, i.e.\ the difference between the two polynomial expressions on both sides of the wall, equals

$$ P_1^\delta - P_2^\delta = \binom{r}{r_1,r_2} \cdot \delta \cdot \mathrm{H}_0(\bfx_I\cup \{\delta\}, \bfe_I) \cdot \mathrm{H}_0(\bfx_{I^c}\cup \{-\delta\}, \bfe_{I^c}).$$
    \end{theorem}

The count of leaky covers satisfying $\psi$ conditions reduces the computation of all descendant leaky double Hurwitz numbers to those where the branch insertion is trivial \eqref{eqn:DRdescendants}. We seek to further improve the situation by shrinking the class of initial conditions needed. We observe that any descendant insertion is equivalent to a linear combination of boundary divisors together with the class $\kappa_1$, thanks to a second splitting formula for the double ramification cycle (essentially proven in \cite{CSS}). It shows that the cycles
    \[
    x_i \psi_i \cdot \DR_g(\bfx) \text{ and } \frac{k}{2g-2+n} \kappa_1 \cdot \DR_g(\bfx) 
    \]
    differ by a sum over graphs with two vertices carrying $\DR$-cycles (see Proposition \ref{pro:DR_psi_splitting}). Using this relation, we can calculate the intersection number \eqref{eqn:DRdescendants} by exchanging one $\psi$-class after the other for either a multiple of $\kappa_1$ or a term supported in the boundary. After performing this procedure $2g-3+n$ times, we are left with a graph sum with vertex terms only involving powers of the class $\kappa_1$.
    
To state a formal recursion we introduce the notation
\begin{equation}  \label{eqn:psi_kappa_descendant}
    \mathrm{H}_g(\bfx, \bfe, f) = \int_{\Mbar_{g,n}} \DR_g(\bfx) \cdot \psi_1^{e_1} \cdots \psi_n^{e_n} \cdot \kappa_1^f\,
\end{equation}
mixing $\psi$-insertions with a power of $\kappa_1$.

\begin{theorem}
 The numbers $\mathrm{H}_g(\bfx, \bfe, f)$ can be recursively computed (see \eqref{eqn:Hxef_recursion}) from the intersection numbers
    \begin{equation} \label{eqn:kappa1intersection}
        \mathrm{H}_{g'}(\bfx', \mathbf{0}, 2g'-3+n') = \int_{\Mbar_{g',n'}} (k \kappa_1)^{2g'-3+n'} \cdot \DR_{g'}(\bfx')\,.
    \end{equation}   
\end{theorem}
Combined with the tropical graph sum formula from Theorem \ref{thm:tropical_formula}, this uniquely determines the $k$-leaky double Hurwitz descendants from the initial data of the numbers \eqref{eqn:kappa1intersection}. These intersection numbers with powers of $\kappa_1$  in turn have been characterized explicitly in upcoming work \cite{Sauvaget_integral} by Sauvaget. In fact, Sauvaget's paper will give a full formula for the integrals $H_g(\bfx, \bfe)$ with $|e|=2g-3+n$ appearing as the vertex multiplicities of our tropical graph sum formula.

Lastly we turn our attention to unexpected vanishings of leaky descendant double Hurwitz numbers. In genus $0$, we can compute such numbers as weighted graph sums where each graph carries a non-negative contribution. 
(This is not true in higher genus, see Example \ref{ex-genus1}.)
This gives a blunt but powerful tool to show non-vanishing, and in fact positivity: it suffices to exhibit a single graph with a non-zero contribution.

If $k=0$, an easy argument shows that $H_{0}(\mathbf{x}, \mathbf{e})>0$ unless $\bfx=0$ and $n> |\mathbf{e}|+3$, see Remark \ref{rem-k=0}.
For the more interesting case $k\neq 0$, we are able to witness, and classify some exotic vanishing  behavior, as summarized in the following theorem.

\begin{theorem}\label{thm-positivity}
Let $g=0$ and $k\neq 0$. Let $|\bfx|=(n-2)k$ and $0\leq |\mathbf{e}| \leq n-3$.

The $k$-leaky descendant $H_{0}(\mathbf{x}, \mathbf{e})$ vanishes if an only if  $k$ is even, $x_i=m_i\cdot \frac{k}{2}$ for $m_i\in \mathbb{N}_{>0}$ and for every subset $I\subset\{1,\ldots,n\}$ we have
$$\sum_{i\in I}e_i<  \sum_{i\in I}m_i - |I|+1.$$ 
In all other cases we have $H_{0}(\mathbf{x}, \mathbf{e})>0$.
\end{theorem}



\subsection{Current landscape and future directions}

The $k$-leaky descendant double Hurwitz numbers are a large family of enumerative invariants rich in combinatorial structure and closely tied to the geometry of double ramification cycles, and have been studied and progressively generalized over the last few decades. In Table \ref{tab:landscape1}  we list the known results on the numbers $\mathrm{H}_g(\bfx, \bfe)$.

\begin{table}[tb]
\hspace{-2.5cm}
\renewcommand{\arraystretch}{2}
\begin{tabular}{c V{3} cc V{3} c|c|c|c }
 &  & & $\Br_g(\bfx)$ & \multicolumn{2}{c|}{$\Br_g^c(\bfx) \cdot \psi^{\bfe}$} & $\psi^{\bfe}$\\
 & &  & & $\bfe = e_1 \delta_1$ & $\bfe$ arbitrary & \\
  \Xhline{3\arrayrulewidth}
  \multirow{4}{*}{\makecell{$g=0$}} &
  \multirow{2}{*}{tot. ram.}   & $k=0$   & \makecell{formula\\ \cite{Goulden1992The-combinatorial-relation}} & \cite{Cavalieri2023-One-part-leaky} & & \multirow{4}{*}{\makecell{classical\\formula\\$\binom{n-2}{\bfe}$}} \\
 \cline{3-6} & & $k>0$      & \makecell{formula\\ \cite{Cavalieri2023-One-part-leaky}} & \cite{Cavalieri2023-One-part-leaky}  & & \\
 \cline{2-6}  & \multirow{2}{*}{arb. ram.}   & $k=0$    & \makecell{wall-crossing\\ \cite{Shadrin2008-Chamber-structure}}  &  & & \\
 \cline{3-6} & & $k>0$      &  &  & \makecell{wall-crossing\\{[CMS24]}} &
  \\
\Xhline{3\arrayrulewidth}
 \multirow{4}{*}{\makecell{$g>0$}}& \multirow{2}{*}{tot. ram.}   & $k=0$  & \makecell{one-part double\\Hurwitz numbers\\ \cite{Goulden2005Towards-the-geo} } & \cite{Cavalieri2023-One-part-leaky} & & see below \\
 \cline{3-7}  & & $k>0$      &  &   & & see below \\
  \cline{2-7} & \multirow{2}{*}{arb. ram.}   & $k=0$    & \makecell{double\\Hurwitz numbers\\wall-crossing\\ \cite{Cavalieri2011-Wall-crossings}} &  & & \makecell{generating\\function\\ \cite{BSSZ}}\\
 \cline{3-7}  & & $k>0$      & \makecell{$k$-leaky double\\Hurwitz numbers\\piecew. polynomiality\\\cite{CMR22} } &  & \makecell{piecewise\\polynomiality\\ {[CMS24]}} & \makecell{generating\\function\\\cite{CSS}\\for $\bfe = e_1\delta_1$,\\\cite{Sauvaget_integral} for all $\bfe$}
\end{tabular}

\caption{The landscape of known results on double Hurwitz descendants $\mathrm{H}_g(\bfx, \bfe)$. Here [CMS24] denotes the present paper.}
\label{tab:landscape1}
\end{table}

Depending on the parameters $g, \bfx, \bfe$ they have been characterized to different extents, ranging from structural properties to explicit formulas:
\begin{itemize}
    \item  \emph{explicit formulas} for (non-leaky) double Hurwitz numbers have been found mostly in the special case where all but one entry of $\bfe$ have the same sign (the case of total ramification, see \cite{Goulden1992The-combinatorial-relation, Goulden2005Towards-the-geo}). In the forthcoming paper \cite{Cavalieri2023-One-part-leaky} we extend these formulas to $k$-leaky double Hurwitz numbers and even to the case of $\psi$-insertions at the marking of total ramification;
    \item  for $k$-leaky double ramification descendants (without branch conditions), explicit \emph{generating functions} have been found in \cite{BSSZ} (for $k=0$) and \cite{CSS} (for $k>0$ and insertions given by a power of a single $\psi$-class). The forthcoming paper \cite{Sauvaget_integral} proves a full formula for the generating function;
    \item the \emph{piecewise polynomiality} of $\mathrm{H}_g(\bfx, \mathbf{0})$ and \emph{wall-crossing formulas} for different values of $\bfx$ were found  in \cite{Shadrin2008-Chamber-structure, Cavalieri2011-Wall-crossings} for $k=0$ and \cite{CMR22} for $k>0$. In our paper we show the piecewise polynomiality in full generality and work out the wall-crossing structure for $g=0$.
\end{itemize}
As illustrated in Table \ref{tab:landscape1}, there is lots of room for progress. In particular, in the total ramification case it seems feasible to hope for an explicit formula in genus $0$, and an approach via generating functions in arbitrary genus.

A second direction of study is the enumerative interpretation of $k$-leaky double Hurwitz numbers. As explained before, for $k=0$ the number $\mathrm{H}_g(\bfx, \mathbf{0})$ counts covers of the projective line with fixed ramification profiles over $0, \infty$ and simple ramification over $2g-2+n$ other points. It is natural to expect that for arbitrary $k$, the number $\mathrm{H}_g(\bfx, \mathbf{0})$ could be a count of $k$-differentials with given zero and pole-orders together with $2g-2+n$ further conditions reducing the dimension to zero.

One approach in this direction is studied by \cite{GendronTahar, BuryakRossiCount, ChenPrado} in the case of $g=0$, $k=1$ and total ramification. More precisely, consider a vector 
$$\bfx = (d, -a_1, \ldots, -a_n) \in \mathbb{Z}^{n+1} \text{ (with }d,a_1, \ldots, a_n >0)$$
such that $|\bfx| = n-1$. In the notation of our paper, the authors study the moduli space $\mathcal{H}_0(\bfx)$ parameterizing tuples
\[
(C, q, p_1, \ldots, p_n, \eta)
\]
of a smooth genus $0$ curve $C$ with $n+1$ distinct marked points and a $k=1$ differentials $\eta$ on $C$ with zeros and poles of orders $x_i-1$ at the marked points. 
Given a fixed vector $\vec r = (r_1, \ldots , r_n) \in \mathbb{C}^{n} \setminus \{0\}$ satisfying $\sum r_i=0$, the subset $\mathcal{H}_0(\bfx)^{\vec r} \subseteq \mathcal{H}_0(\bfx)$ where the differential $\eta$ has residue $r_i$ at $p_i$ is a finite set. The papers above count the number of points in dependence of $\vec r$. As an example, for a general vector $\vec r$ the count is given by the formula
\[
|\mathcal{H}_0(\bfx)^{\vec r}| = (d-1) \cdot (d-2) \cdots (d-(n-2)).
\]
On the other hand, in \cite{Cavalieri2023-One-part-leaky} we show that the $k$-leaky double Hurwitz number is given by
\[
\mathrm{H}_0(\bfx) = (n-1)! (d-\frac{1}{2})\cdot (d-\frac{2}{2}) \cdots (d-\frac{n-2}{2}).
\]
While the formulas do not agree, they bear a strong resemblance and share structural properties (such as being a polynomial of degree $n-2$ in the entries of $\bfx$). It seems interesting to explore whether the setup in \cite{GendronTahar, BuryakRossiCount, ChenPrado} can be modified to give an enumerative interpretation for the number $\mathrm{H}_0(\bfx)$. This could also show a path to extending their counting problem to higher genus, where fixing the residues at poles no longer cuts the dimension to zero.

\subsection{Acknowledgments}
We would like to thank Dawei Chen, Sam Molcho, Adrien Sauvaget for useful discussions. We are especially grateful to Dhruv Ranganathan for his substantive contributions during the early stages of the project.

Part of this work was pursued during the first author's academic visit at ETH's  Institute for Mathematical Research (FIM), whose support is gratefully acknowledged. The first author was supported by NSF DMS-2100962.
The second author acknowledges support by the Deutsche Forschungsgemeinschaft (DFG, German Research Foundation), Project-ID 286237555, TRR 195.

\section{Background}

\subsection{Tropical Curve Counts}
{An \emph{abstract tropical} \emph{curve} is a connected metric graph $\Gamma$, such that edges leading to leaves (called \emph{ends}) have infinite length, together with a genus function $g:\Gamma\rightarrow \ZZ_{\geq 0}$ with finite support. Locally around a point $p$, $\Gamma$ is homeomorphic to a star with $r$ halfrays. 
The number $r$ is called the \emph{valence} of the point $p$ and denoted by $\val(p)$. The \emph{minimal vertex set} of $\Gamma$ is defined to be the points where the genus function is nonzero, together with points of valence different from $2$. The vertices of valence greater than $1$ are called  \textit{inner vertices}. Besides \emph{edges}, we introduce the notion of \emph{flags} of $\Gamma$. A flag is a pair $(v,e)$ of a vertex $v$ and an edge $e$ incident to it ($v\in \partial e$). Edges that are not ends are required to have finite length and are referred to as \emph{bounded} or \textit{internal} edges.

A \emph{marked tropical curve} is a tropical curve whose leaves are labeled. An isomorphism of a tropical curve is an isometry respecting the leaf markings and the genus function. The \emph{genus} of a tropical curve $\Gamma$ is given by
\[
g(\Gamma) = h_1(\Gamma)+\sum_{p\in \Gamma} g(p).
\]
A curve of genus $0$ is called \emph{rational} and a curve satisfying $g(p)=0$ for all $p\in \Gamma$ is called \emph{explicit}. The \emph{combinatorial type} is the equivalence class of tropical curves obtained by identifying any two tropical curves which differ only by edge lengths.

We want to examine covers of $\RR$ by graphs up to additive translation,  and  equip  $\RR$ with a polyhedral subdivision to ensure the result is a map of metric graphs (see e.g.\ Section 5.4 and \ Figure 3 in \cite{MW17}). A \textit{metric line graph} is any metric graph obtained from a polyhedral subdivision of $\RR$. The metric line graph determines the polyhedral subdivision up to translation. We fix an orientation of a metric line graph going from left to right (i.e.\ from negative values in $\RR$ to positive values).

\begin{definition}[Leaky cover, \cite{CMR22}] \label{Def:leakycov}
Let $\pi:\Gamma\rightarrow T$ be a surjective map of metric graphs where $T$ is a metric line graph. We require that $\pi$ is piecewise integer affine linear, the slope of $\pi$ on a flag or edge $e$ is a positive integer called the \emph{expansion factor} $\omega(e)\in \NN_{> 0}$. 

For a vertex $v\in \Gamma$, the \emph{left (resp.\ right) degree of $\pi$ at $v$} is defined as follows. Let $f_l$ be the flag of $\pi(v)$ in $T$ pointing to the left and $f_r$ the flag pointing to the right. Add the expansion factors of all flags $f$ adjacent to $v$ that map to $f_l$ (resp.\ $f_r$):
\begin{equation}
d_v^l=\sum_{f\mapsto f_l} \omega(f), \;\;\;\; d_v^r=\sum_{f\mapsto f_r} \omega(f).
\end{equation} 

\noindent
The map $\pi:\Gamma\rightarrow T$ is called a \emph{$k$-leaky cover} if for every $v\in \Gamma$
\begin{equation}\label{eq:leakybal}d_v^l-d_v^r= k(2g(v)-2+\val(v)).\end{equation}
\end{definition}

\begin{remark}[Vertex sets]\label{rem-vertexsets}
We impose a stability condition: A $k$-leaky cover $\Gamma\to T$ is called stable if the preimage of every vertex of $T$ contains a vertex of $\Gamma$ in its preimage which is of genus greater than $0$ or valence greater than $2$. 

Furthermore, we often stabilize the source tropical curve further by passing to its minimal vertex set (containing only the points where the genus function is nonzero, together with points of valence different from $2$). In that way, we lose the property that the cover is a map of graphs, however, this vertex structure is relevant to determine valencies correctly for the purpose of Psi-conditions.
\end{remark}

\begin{definition}[Left and right degree]
The \emph{left (resp.\ right) degree}  of a leaky cover is the tuple of expansion factors of its ends mapping asymptotically to $-\infty$ (resp.\ $+\infty$). 
The tuple is indexed by the labels of the ends mapping to  $-\infty$ (resp.\ $+\infty$).  When the order imposed by the labels of the ends plays no role, we drop the information and treat the left and right degree only as a multiset. 
\end{definition}
By convention, we denote the left degree by $\mathbf{x}^{+}$ and the right degree by $\mathbf{x}^{-}$. In the right degree, we use negative signs for the expansion factors, in the left degree positive signs. We also merge the two to one vector which we denote $\bfx=(x_1,\ldots,x_n)$  called the \emph{degree}. The labeling of the ends plays a role: the expansion factor of the end with the label $i$ is $x_i$. In $\mathbf{x}$, we  distinguish the expansion factors of the left ends from those of the right ends by their sign.
An Euler characteristic calculation, combined with the leaky cover condition, shows that 
$$ \sum_{i=1}^n x_i=k\cdot(2g-2+n), $$
where $g$ denotes the genus of $\Gamma$.

An automorphism of a leaky cover is an automorphism of $\Gamma$ compatible with $\pi$.}


\subsection{Insertions from the stack of expansions}
The stack $\mathsf{Ex}$ of expansions has been studied extensively in the context of Gromov-Witten theory (see \cite{Li2001Stable-morphism, Li2002A-degeneration-, Abramovich2013Expanded-degene}). 
In Definition \ref{def-hgxe}, we defined the branch class $\Br_g^c(\bfx)$ as a pull-back of a codimension $c$ class on $\mathsf{Ex}$ under a map $\mathsf{t} : \widehat{\mathcal{M}}^\bfx_{g,n} \to \mathsf{Ex}$. We start by giving a concrete description of this map $\mathsf{t}$ and its claimed properties from the introduction.

\begin{proposition} \label{pro:tmap}
Let $\widehat{\mathcal{M}}^\bfx_{g,n}$ be the blowup of $\Mbar_{g,n}$ associated to the vector $\bfx$ and a small non-degenerate stability condition (see \cite[Section 4]{Holmes_Log_DR}). Then it admits a map $\mathsf{t} : \widehat{\mathcal{M}}^\bfx_{g,n} \to \mathsf{Ex}$ to the stack of expansions such that for the inclusion $\iota: \Mbar^{\sim}_{g,\bfx} \to \widehat{\mathcal{M}}^\bfx_{g,n}$ of the log double-ramification locus, we have
\begin{itemize}
    \item the diagram in Figure \ref{fig:maptoex} commutes,
    \item the tropicalization of the composition $\mathsf{t} \circ \iota$ is the map $\leak\to\mathsf{tEx}$ sending a tropical cover $\Gamma \to \mathbb{R}$ to the induced subdivision of its target $\mathbb{R}$ at the images of vertices of $\Gamma$.
\end{itemize}
\end{proposition}
\begin{proof}
The tropicalization $\widehat \Sigma_{g,n}^\bfx$ of the space $\widehat{\mathcal{M}}^\bfx_{g,n}$ associated to a small, nondegenerate stability condition $\theta$ has been described in \cite[Section 4.2.2]{Holmes_Log_DR}. Its cones are indexed by tuples $(\widehat \Gamma, D, I)$ where $\widehat \Gamma$ is a quasi-stable graph, $D$ is a $\theta$-stable divisor on $\widehat \Gamma$ and $I$ is an acyclic flow on $\widehat \Gamma$ with $\mathrm{div}(I) = \underline{deg}((\omega^\textup{log})^k(-\sum a_i x_i)) - D$. 
The cone $\sigma_{(\widehat \Gamma, D, I)}$ associated to this tuple parameterizes tropical covers from (a tropical curve with underlying graph) $\widehat \Gamma$ to $\mathbb{R}$ having slope $I(e)$ on each edge $e$ of $\widehat \Gamma$. Taking the image of the stable vertices of $\widehat \Gamma$ in $\mathbb{R}$ defines a subdivision of the real line, i.e. an element of the tropicalization $\mathsf{tEx}$. One can verify that this operation defines a morphism
\begin{equation} \label{eqn:t_trop_def}
\widehat \Sigma_{g,n}^\bfx \to \mathsf{tEx}
\end{equation}
of cone stacks. Since $\mathsf{Ex} = \mathcal{A}_{\mathsf{tEx}}$ is its own Artin fan, we can define the map
\[
\mathsf{t} : \widehat{\mathcal{M}}^\bfx_{g,n} \to \mathcal{A}_{\widehat \Sigma_{g,n}^\bfx} \to \mathcal{A}_{\mathsf{tEx}} = \mathsf{Ex}
\]
as the composition of the map from $\widehat{\mathcal{M}}^\bfx_{g,n}$ to its Artin fan, with the morphism of Artin fans induced from the cone stack map \eqref{eqn:t_trop_def}. The tropicalization $\leak$ is the sub-complex of $\widehat \Sigma_{g,n}^\bfx$ where the curve $\widehat \Gamma$ is stable, and on there the map to $\mathsf{tEx}$ continues to record the edge lengths of the subdivided target 
$\mathbb{R}$. This shows the second bullet point, and for the first we simply observe that two maps to $\mathsf{Ex}$ are equal if and only if their tropicalizations coincide (again since $\mathsf{Ex}$ is its own Artin fan).
\end{proof}

In our paper, the only type of insertion from $\CH_*(\mathsf{Ex})$ that we consider is the fundamental class $[T_c] \in \CH^c(\mathsf{Ex})$ of the codimension $c$ boundary stratum of $\mathsf{Ex}$. Further natural insertions would be the cotangent line classes $\Psi_0, \Psi_\infty$ of the expanded target at $0, \infty$. For $\pi: \widehat{\mathcal{M}}^\bfx_{g,n} \to \Mbar_{g,n}$ the blow-up on which $\logDR_g(\bfx)$ is supported, a conjectural description for the class
\[
\pi_*\left(\logDR_g(\bfx) \cdot \mathsf{t}^* \Psi_\infty^u \right) \in \CH^{g+u}(\Mbar_{g,n})
\]
was proposed in \cite[Conjecture 1.4]{2022arXiv221204704C}. This formula would allow to calculate the intersection numbers of $\logDR_g(\bfx)$ against both powers of $\Psi_\infty$ and further $\psi$-classes in examples. However, for now we restrict our attention to the insertions $[T_c]$ mentioned above.

\section{Tropical leaky descendants}

\begin{definition}[Psi-conditions for leaky covers] \label{def:psi}
Let $g,n \geq 0$ such that $2g-2+n > 0$ and consider vectors $\bfx \in \mathbb{Z}^n$ such that $|\bfx| = k (2g-2+n)$ for some $k \in \mathbb{Z}$. Let $\bfe \in \mathbb{Z}_{\geq 0}^n$ such that $0 \leq |\bfe| \leq 2g-3+n$.

Let $\pi:\Gamma\rightarrow T$ be a $k$-leaky cover.
For a vertex $v$, let $I_v\subset \{1,\ldots,n\}$ be the subset of ends adjacent to $v$ after passing to the minimal vertex set of $\Gamma$ (see Remark \ref{rem-vertexsets}).

  We say that $\pi:\Gamma\rightarrow T$  {\it satisfies the Psi-conditions} $\mathbf{e}$
  if for all vertices $v$ of $\Gamma$,
  \begin{equation} \label{eq:psicondition}
  \val(v)=\sum_{i\in I_v}e_i+3-2g(v).\end{equation}
\end{definition}

\begin{example}
    For an example of three $1$-leaky tropical covers of genus $1$ and degree $\bfx=(7,-3,-1)$ satisfying the Psi-conditions $\bfe=(1,0,0)$, see Figure \ref{fig:exleakyPsicover}.
\end{example}

\begin{figure}[tb]
    \centering

\tikzset{every picture/.style={line width=0.75pt}} 

\begin{tikzpicture}[x=0.75pt,y=0.75pt,yscale=-1,xscale=1]

\draw    (120,310) -- (180,310) ;
\draw    (180,310) .. controls (208,300.22) and (234.5,305.72) .. (250,320) ;
\draw    (180,310) .. controls (199.5,323.72) and (225,329.22) .. (250,320) ;
\draw    (180,310) .. controls (204,294.22) and (312,299.22) .. (330,300) ;
\draw    (250,320) -- (330,320) ;
\draw    (120,390) -- (180,390) ;
\draw  [fill={rgb, 255:red, 0; green, 0; blue, 0 }  ,fill opacity=1 ] (180,390) .. controls (180,388.76) and (181.01,387.75) .. (182.25,387.75) .. controls (183.49,387.75) and (184.5,388.76) .. (184.5,390) .. controls (184.5,391.24) and (183.49,392.25) .. (182.25,392.25) .. controls (181.01,392.25) and (180,391.24) .. (180,390) -- cycle ;
\draw    (182.25,390) -- (250,390) ;
\draw    (250,390) -- (330,370) ;
\draw    (250,390) -- (330,410) ;
\draw    (120,242) -- (180,242) ;
\draw    (180,242) .. controls (208,232.22) and (234.5,237.72) .. (250,252) ;
\draw    (180,242) .. controls (199.5,255.72) and (225,261.22) .. (250,252) ;
\draw    (180,242) .. controls (204,226.22) and (312,231.22) .. (330,232) ;
\draw    (250,252) -- (330,252) ;

\draw (339,312) node [anchor=north west][inner sep=0.75pt]   [align=left] {$\displaystyle 3$};
\draw (339,290) node [anchor=north west][inner sep=0.75pt]   [align=left] {$\displaystyle 1$};
\draw (249,302) node [anchor=north west][inner sep=0.75pt]   [align=left] {$\displaystyle 1$};
\draw (211,330) node [anchor=north west][inner sep=0.75pt]   [align=left] {$\displaystyle 3$};
\draw (141,292) node [anchor=north west][inner sep=0.75pt]   [align=left] {$\displaystyle 7$};
\draw (339,402) node [anchor=north west][inner sep=0.75pt]   [align=left] {$\displaystyle 3$};
\draw (341,360) node [anchor=north west][inner sep=0.75pt]   [align=left] {$\displaystyle 1$};
\draw (221,370) node [anchor=north west][inner sep=0.75pt]   [align=left] {$\displaystyle 5$};
\draw (131,370) node [anchor=north west][inner sep=0.75pt]   [align=left] {$\displaystyle 7$};
\draw (339,244) node [anchor=north west][inner sep=0.75pt]   [align=left] {$\displaystyle 3$};
\draw (339,222) node [anchor=north west][inner sep=0.75pt]   [align=left] {$\displaystyle 1$};
\draw (249,234) node [anchor=north west][inner sep=0.75pt]   [align=left] {$\displaystyle 2$};
\draw (211,262) node [anchor=north west][inner sep=0.75pt]   [align=left] {$\displaystyle 2$};
\draw (141,224) node [anchor=north west][inner sep=0.75pt]   [align=left] {$\displaystyle 7$};

\begin{scope}[yshift=-60pt] 

\draw    (120,538) -- (180,538) ;
\draw    (180,538) .. controls (208,528.22) and (234.5,533.72) .. (250,548) ;
\draw    (180,538) .. controls (199.5,551.72) and (225,557.22) .. (250,548) ;
\draw    (180,538) .. controls (204,522.22) and (312,527.22) .. (330,528) ;
\draw    (250,548) -- (330,548) ;

\draw (339,540) node [anchor=north west][inner sep=0.75pt]   [align=left] {$\displaystyle 1$}; 
\draw (339,518) node [anchor=north west][inner sep=0.75pt]   [align=left] {$\displaystyle 3$}; 
\draw (249,530) node [anchor=north west][inner sep=0.75pt]   [align=left] {$\displaystyle 1$};
\draw (211,558) node [anchor=north west][inner sep=0.75pt]   [align=left] {$\displaystyle 1$};
\draw (141,520) node [anchor=north west][inner sep=0.75pt]   [align=left] {$\displaystyle 7$};

\end{scope}

\begin{scope}[yshift=0pt] 

\draw    (120,538) -- (180,538) ;

\draw    (180,538) .. controls (204,522.22) and (312,527.22) .. (330,528) ;

\draw    (180,538) .. controls (204,522.22) and (312,527.22) .. (330,528) ;

\draw    (180,538) -- (330,548) ;
\draw    (180,538) -- (250,565) ;

\draw (339,540) node [anchor=north west][inner sep=0.75pt]   [align=left] {$\displaystyle 1$}; 
\draw (339,518) node [anchor=north west][inner sep=0.75pt]   [align=left] {$\displaystyle 3$}; 
\draw (211,558) node [anchor=north west][inner sep=0.75pt]   [align=left] {$\displaystyle 1$};
\draw (141,520) node [anchor=north west][inner sep=0.75pt]   [align=left] {$\displaystyle 7$};
\draw[fill] (250,565) circle (1.5pt);

\end{scope}

\end{tikzpicture}

    \caption{Five $1$-leaky tropical covers of genus $1$ and degree $\bfx=(7,-3,-1)$ satisfying the Psi-conditions $\bfe=(1,0,0)$. The vertices marked with a dot are vertices of genus $1$. All other vertices are of genus $0$. All five pictures cover a line graph $T$ with $2$ vertices. We did not specify lengths in the picture, as the lengths in the source graph $\Gamma$ are determined by the expansion factors and the lengths in $T$.}
    \label{fig:exleakyPsicover}
\end{figure}
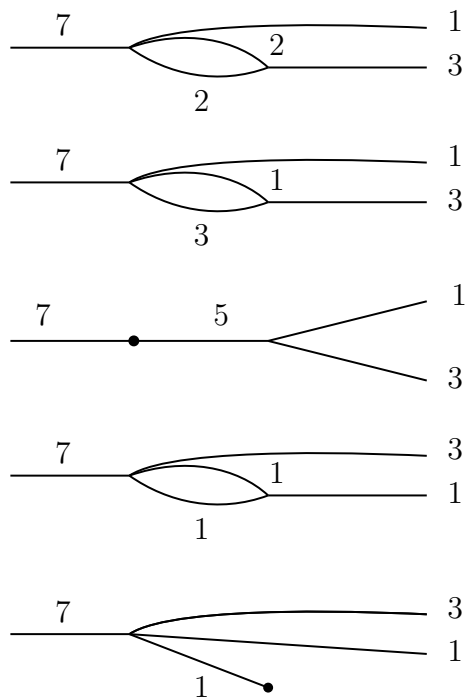

\begin{definition}[Vertex multiplicities]\label{def-multvertex}
Let $\pi:\Gamma\rightarrow T$ be a $k$-leaky cover satisfying the Psi-conditions $\bfe$.
For a vertex $v$, let $I_v\subset \{1,\ldots,n\}$ be the subset of ends adjacent to $v$ after passing to the minimal vertex set of $\Gamma$ (see Remark \ref{rem-vertexsets}).
   Let $\mathbf{x}(v)$ denote the vector containing the (left and right) local degree of $v$, let $g(v)$ denote the genus of $v$.

   We define the \emph{vertex multiplicity} to be

   $$\mult_v:= \int_{\overline{M}_{g(v),\val(v)}} \DR_{g(v)}(\mathbf{x}(v))\cdot \prod_{i\in I_v} \psi_i^{e_i}$$
\end{definition}

\begin{example}\label{ex-vertexmult}
    For the $3$-valent genus $0$ vertices in the covers of Figure \ref{fig:exleakyPsicover}, the vertex multiplicity is $1$. This is true because in genus $0$, $\DR$ equals $\Mbar_{0,3}$ which is just a point. For the $4$-valent vertices of genus $0$, it is still true that $\DR$ equals $\Mbar_{0,4}$, but now we take the integral over $\psi_1$. This is also just a point, so again, these vertex multiplicities are $1$. The genus $1$ vertex may be evaluated using the software {\tt admcycles} \cite{admcycles} to obtain 
    $\int_{\Mbar_{1,2}} \DR_1(7,-5)\cdot \psi_1 = 35/24$.
\end{example}



\begin{definition}[Count of $k$-leaky covers satisfying Psi-conditions]\label{def-trophgxe}
   Let $g,n \geq 0$ such that $2g-2+n > 0$ and consider vectors $\bfx \in \mathbb{Z}^n$ such that $|\bfx| = k (2g-2+n)$ for some $k \in \mathbb{Z}$. Let $\bfe \in \mathbb{Z}_{\geq 0}^n$ such that $0 \leq |\bfe| \leq 2g-3+n$, and  $b=2g-3+n-|\bfe|\geq 0$.


We define
\begin{equation}
  \mathrm{H}^{\trop}_g(\bfx, \bfe) = \sum_\pi \frac{1}{|\Aut(\pi)|}\cdot \prod_e \omega(e) \cdot \prod_v \mult_v,  
\end{equation}
where:
\begin{itemize}
    \item $\pi:\Gamma\rightarrow T$ ranges among all leaky covers of degree $\mathbf{x}$ and genus $g$ (Definition \ref{Def:leakycov}) and satisfying the Psi-conditions $\mathbf{e}$ (Definition \ref{def:psi});
    \item the first product is over all bounded edges of $\Gamma$ (according to its minimal vertex set, see Remark \ref{rem-vertexsets});
    \item the second product goes over the set of vertices of $\Gamma$ and $\mult_v$ is as in Definition \ref{def-multvertex}.
\end{itemize}

\end{definition}

\begin{example}\label{ex-count}
Fix the leaking $1$, $n=3$ ends, genus $1$, degree $\bfx=(7,-3,-1)$ and the Psi-conditions $\bfe=(1,0,0)$. Then the  covers we have to consider are  depicted in Figure \ref{fig:exleakyPsicover}. 
  For these three $1$-leaky covers  we obtain the following multiplicities, using Example \ref{ex-vertexmult} discussing vertex multiplicities: 

\begin{center}
\begin{tabular}{|c|c|c|c|c|}
\hline $\pi_i$    &  $1/|\Aut(\pi_i)$ & $\prod_e \omega(e)$ & $\mult_v$ & $\mult(\pi_i)$ \\
\hline $\pi_1$    & $1/2$ & $4$ & $(1,1)$ & $2$\\
\hline $\pi_2$    & $1$ & $3$ & $(1,1)$ & $3$\\
\hline $\pi_3$    & $1$ & $5$ & $(35/24,1)$ & $175/24$\\
\hline $\pi_4$    & $1/2$ & $1$ & $(1,1)$ & $1/2$\\
\hline $\pi_5$    & $1$ & $1$ & $(1,-1/24)$ & $-1/24$\\
\hline
\end{tabular}
\end{center}
In total we obtain
\[
H_1^{\trop}(\bfx,\bfe) = 2 + 3 + \frac{175}{24} + \frac{1}{2} - \frac{1}{24} = \frac{51}{4}.
\]
  
    


\end{example}

\begin{theorem}[Correspondence Theorem for leaky covers with Psi-conditions]\label{thm-corres}

Let $g,n \geq 0$ such that $2g-2+n > 0$ and consider vectors $\bfx \in \mathbb{Z}^n$ such that $|\bfx| = k (2g-2+n)$ for some $k \in \mathbb{Z}$. Let $\bfe \in \mathbb{Z}_{\geq 0}^n$ such that $0 \leq |\bfe| \leq 2g-3+n$.

Then the $k$-leaky double Hurwitz descendant (defined in \ref{def-hgxe}) equals the count of tropical $k$-leaky covers satisfying Psi-conditions (defined in \ref{def-trophgxe}):

$$\mathrm{H}^{\trop}_g(\bfx, \bfe)=\mathrm{H}_g(\bfx, \bfe).$$

\end{theorem}

\begin{proof}
We begin the proof by giving a geometric interpretation of the product
\begin{equation} \label{eq:lDRBR}
   \logDR_g(\mathbf{x})\cdot Br^c_g(\mathbf{x})\in \logCH^{g+c}(\Mbar_{g,n}) 
\end{equation}
of the logarithmic double ramification cycle with the {\it branch class}, i.e. the piecewise polynomial function $Br^c_g(\mathbf{x}) = \mathsf{t}^\ast[T_c]$.
The intersection class in \eqref{eq:lDRBR} is represented by the total codimension $c$ boundary of $\logDR_g(\mathbf{x})$, i.e. the sum of all boundary strata $\Delta^{\mathbf x}_\Gamma$ of codimension $c$. 
First we discuss which strata may have a non-zero intersection with the monomial $\psi^{\mathbf e} = \psi_1^{e_1}\ldots \psi_n^{e_n}$. Since $\psi$ classes are pulled back via the contraction morphism to $\Mbar_{g,n}$ we may apply projection formula:
\begin{equation}
\int_{\widehat{\mathcal{M}}^\bfx_{g,n}} \Delta^{\mathbf x}_\Gamma \cdot q^\ast(\psi^{\mathbf e})= \int_{\Mbar_{g,n}} q_\ast (\Delta^{\mathbf x}_\Gamma) \cdot \psi^{\mathbf e}.
\end{equation}

The pushforward $q_\ast(\Delta^{\mathbf x}_\Gamma)$ vanishes because of positive dimensional fibers unless the dual graph of the source graph has exactly $c+1$ vertices, each mapping to one of the vertices of the line graph of $T_c$.
In this case  we have:
\begin{equation}\label{eq:pushfwd}
    q_\ast(\Delta^{\mathbf x}_\Gamma) = m_\Gamma \
    \xi_{\Gamma\ \ast}\left(
    \prod_{v\in V(\Gamma)} \DR_{g(v)}(\mathbf{x}(v))\right),
\end{equation}
where the multiplicity $m_\Gamma$ is obtained as the product of the expansion factors on the bounded edges of $\Gamma$ divided by the order of the automorphism group of the decorated graph $\Gamma$. The vector of integers $\mathbf{x}(v)$ consists of the signed expansion factors of the half edges incident to the vertex $v$ of the graph $\Gamma$.

For $i= 1, \ldots , n$ denote by $v(i)$ the vertex of $\Gamma$
to which the $i$-th marked end is incident, and $\pi_{v(i)}$ the projection from the product $\prod_{v\in V(\Gamma)} \Mbar_{g(v), n(v)}$ onto the corresponding factor. It is a standard fact of psi classes (see e.g. \cite[Chapter XX, equation (4.30)]{ACG_Vol2}) that
\begin{equation}\label{eq:pbpsi}
  \xi_{\Gamma}^\ast \psi_i =  \pi_{v(i)}^\ast \psi_i.
\end{equation}
Defining $I_v$ to be the set of $i$ such that $v(i) = v$ and combining \eqref{eq:pushfwd} and \eqref{eq:pbpsi} we obtain:
\begin{equation}\label{eq:intersectpsidr}
 q_\ast (\Delta^{\mathbf x}_\Gamma) \cdot \psi^{\mathbf e} =   m_\Gamma \
    \xi_{\Gamma\ \ast}\left(
    \prod_{v\in V(\Gamma)} \pi_v^\ast\left(\prod_{i\in I_v} \psi_i^{e_i}\right)\DR_{g(v)}(\mathbf{x}(v))\right).  
\end{equation}
Integrating \eqref{eq:intersectpsidr} one obtains:
\begin{equation}\label{eq:integrated}
 \int_{\Mbar_{g,n}} q_\ast (\Delta^{\mathbf x}_\Gamma) \cdot \psi^{\mathbf e} =   m_\Gamma \ \prod_{v\in V(\Gamma)} \int_{\Mbar_{g(v),n(v)}}
 \DR_{g(v)}(\mathbf{x}(v))
 \prod_{i\in I_v}  \psi_i^{e_i}.     
\end{equation}

The integrals on the right hand side of $\eqref{eq:integrated}$  vanish by dimension reasons unless 
$n(v)=\sum_{i\in I_v}e_i+3-2g(v).$
In conclusion, the codimension $c$
 strata of $\logDR_g(\mathbf{x})$ that do not vanish when intersected with $\psi^{\mathbf{e}}$ are indexed by dual graphs corresponding to the leaky covers from Definition \ref{Def:leakycov}. For any one of this strata, the multiplicity of intersection given by the right hand side of \eqref{eq:integrated} equals the leaky cover multiplicity from \ref{def-trophgxe}. It follows that 
$\mathrm{H}^{\trop}_g(\bfx, \bfe)=\mathrm{H}_g(\bfx, \bfe)$, thus concluding the proof of the theorem. 
\end{proof}


\section{Splitting formulas for (logarithmic) double ramification cycles}
Below we discuss how to convert products of double ramification cycles with certain linear combinations of $\kappa$- and $\psi$-classes into a sum of boundary terms described via further double ramification cycles. These so-called splitting formulas generalize analogous results that first appeared in \cite{BSSZ, CSS}, and allow us to recursively compute double ramification descendants in terms of intersection numbers of $\DR_g(\bfx)$ against powers of $\kappa_1$ in Section \ref{Sect:drdescendant_recursion}.


\subsection{Splittings of \texorpdfstring{$\psi$}{psi}-classes}
\begin{definition}
Let $\pi : \Gamma \to T$ be a stable $k$-leaky cover with expansion factors $\omega(e)$ on its flags or edges $e$. For each vertex $v \in V(\Gamma)$ let $f_l(v), f_r(v)$ be the left and right flag of the vertex. Consider the gluing map
\begin{equation} \label{eqn:gluing_map}
    \xi_\Gamma : \Mbar_\Gamma = \prod_{v \in V(\Gamma)} \Mbar_{g(v), n(v)} \to \Mbar_{g,n}
\end{equation}
associated to the underlying stable graph of $\Gamma$. Then we define
\begin{equation}
    \DR_\pi = (\xi_\Gamma)_* \prod_{v \in V(\Gamma)} \pi_v^* \DR_{g(v)}\left((\omega(f))_{f \mapsto f_l(v)}, (-\omega(f))_{f \mapsto f_r(v)} \right) \in \CH^*(\Mbar_{g,n})\,,
\end{equation}
where $\pi_v : \Mbar_\Gamma \to \Mbar_{g(v), n(v)}$ is the projection to the factor associated to $v \in V(\Gamma)$. Furthermore we define
\begin{equation}
   \mult(\pi) =  \frac{1}{|\Aut(\pi)|}\cdot \prod_e \omega(e) \in \mathbb{Q}
\end{equation}
as the product of the expansion factors at the edges of $\Gamma$, weighted by the number of automorphisms of $\pi$.
\end{definition}
For $\ell \geq 0$ let $T_\ell = T_\ell(\vec w)$ be the metric line graph obtained by subdividing $\mathbb{R}$ at the $\ell+1$ vertices $w_0 < w_1 <  \ldots < w_\ell$. 
\begin{proposition} \label{pro:DR_psi_splitting}
Let $g,n \geq 0$ with $2g-2+n>0$ and $\bfx \in \ZZ^n$ with $|\bfx|=k(2g-2+n)$. Then for any $1 \leq s \leq n$ we have
\begin{equation} \label{eqn:psi_splitting}
\left(x_s \psi_s - \frac{k}{2g-2+n} \kappa_1 \right) \cdot \DR_g(\bfx) = \sum_{\pi: \Gamma \to T_1} \frac{\rho(\pi,s)}{2g-2+n} \cdot \mult(\pi) \cdot \DR_\pi\,,
\end{equation}
where the sum runs over stable $k$-leaky covers $\pi: \Gamma \to T_1$ with precisely two vertices $v_0, v_1$ of $\Gamma$ which map to the two vertices $w_0, w_1$ of $T_1$ (in that order), and
\[
\rho(\pi, s) = \begin{cases}
    2g(v_1)-2+n(v_1)&\text{if }s\text{ adjacent to }v_0,\\
    -(2g(v_0)-2+n(v_0))&\text{if }s\text{ adjacent to }v_1.
\end{cases}
\]
\end{proposition}
\begin{proof}
For $k=0$ and $a_s \neq 0$ this statement was proven in \cite[Theorem 4]{BSSZ}. To prove the general case, consider the vector $\widehat \bfx = (x_1, \ldots, x_n, k)$ associated to a double ramification cycle with an additional free marking. For the forgetful map $F: \Mbar_{g,n+1} \to \Mbar_{g,n}$ we have
\[
F^* \DR_g(\bfx) = \DR_g(\widehat \bfx) \in \CH^g(\Mbar_{g,n+1})\,.
\]
Applying \cite[Proposition 3.1]{CSS} to this extended double ramification cycle (with $s=s, t=n+1$ in the notation of \cite{CSS}) we obtain
\begin{align} \label{eqn:psihat_splitting}
(x_s \psi_s - k \psi_{n+1}) \cdot F^* \DR_g(\bfx) = \sum_{\widehat \pi: \Gamma \to T_1} f_{s,n+1}(\widehat \pi) \cdot \mult(\widehat \pi) \cdot \DR_{\widehat \pi}\,,
\end{align}
where $\widehat \pi$ runs over $(n+1)$-pointed $k$-leaky covers with exactly one vertex $v_0, v_1$ over each of the vertices $w_0, w_1 \in T_1$ and
\[
f_{s,t}(\widehat \pi) = \begin{cases}
    0 & \text{ if $s$ and $n+1$ are adjacent to the same vertex},\\
    1 & \text{ if $s$ is adjacent to $v_0$ and $n+1$ to $v_1$},\\
    -1 & \text{ otherwise}.
\end{cases}
\]
We claim that equation \eqref{eqn:psi_splitting} follows by multiplying boths sides of \eqref{eqn:psihat_splitting} with $\psi_{n+1}$ and pushing forward under the forgetful map $F$. Indeed, using that
\[
F_*(\psi_s \psi_{n+1}) = (2g-2+n) \cdot \psi_s \text{ and } F_*(\psi_{n+1}^2) = \kappa_1\,,
\]
one sees that applying $F_*(\psi_{n+1} \cdot -)$ to the left-hand side of \eqref{eqn:psihat_splitting} gives $(2g-2+n)$ times the left-hand side of \eqref{eqn:psi_splitting}.

For the comparison of the right-hand sides, we first note that for any stable $k$-leaky cover $\widehat \pi$ such that marking $n+1$ lies on a vertex $v$ with $g(v)=0, n(v)=3$ that becomes unstable under forgetting $n+1$, we have $\psi_{n+1} \cdot \DR_{\widehat \pi} = 0$ for dimension reasons. The remaining covers $\widehat \pi$ appearing in the summation are in bijective correspondence to the tuples $(\pi, v_{n+1})$ recording the cover $\pi$ obtained by forgetting marking $n+1$ and the choice $v_{n+1} \in \{v_0, v_1\}$ of the vertex where this marking was attached. Indeed, the fact that marking $n+1$ carries weight $k$ precisely means that its position on the graph $\Gamma$ does not influence the balancing condition on the two vertices. Note also that $\mult(\widehat \pi) = \mult(\pi)$ is preserved under this correspondence.

To conclude the proof, observe that for any cover $\pi$ appearing on the right-hand side of \eqref{eqn:psi_splitting} there is precisely one choice of $v_{n+1}$ such that the corresponding lift $\widehat \pi = (\pi, v_{n+1})$ satisfies $f_{s,n+1}(\widehat \pi) \neq 0$ (namely $v_{n+1}=v_1$ for $s$ adjacent to $v_0$ in $\pi$, and $v_{n+1}=v_0$ otherwise). Then indeed the map $F_*(\psi_{n+1} \cdot -)$ sends the right-hand side of \eqref{eqn:psihat_splitting} to $(2g-2+n)$ times the right-hand side of \eqref{eqn:psi_splitting}. Here the sign of the factor $\rho(\pi,s)$ comes from $f_{s,n+1}(\widehat \pi)$ and its absolute value $2g(v_{n+1})-2+n(v_{n+1})$ comes from the forgetful pushforward of the class $\psi_{n+1}$ on the vertex $v_{n+1}$.
\end{proof}

\begin{remark}
It is an interesting question how to lift equation \eqref{eqn:psi_splitting} to a splitting formula for the logarithmic double ramification cycle $\logDR_g(\bfx)$. We expect that the right-hand side of \eqref{eqn:psi_splitting} generalizes by allowing arbitrary stable $k$-leaky covers $\pi$, with the associated contribution $\logDR_\pi$ given as a suitable log-boundary pushforward of logarithmic double ramification cycles on the vertices of $\Gamma$. The associated language of log-boundary pushforwards is currently being developed in \cite{LogTaut}.
\end{remark}

\subsection{Recursions for double ramification descendants}  \label{Sect:drdescendant_recursion}
Given a vector $\bfx \in \mathbb{Z}^n$ with $|\bfx|=k(2g-2+n)$ and $\bfe \in \mathbb{Z}_{\geq 0}^n$ with $|\bfe|=2g-3+n$, we want to give a recursion determining all intersection numbers
\begin{equation}  \label{eqn:psi_descendant}
    \mathrm{H}_g(\bfx, \bfe) = \int_{\Mbar_{g,n}} \DR_g(\bfx) \cdot \psi_1^{e_1} \cdots \psi_n^{e_n}\,.
\end{equation}
Note that a priori, the invariant $\mathrm{H}_g(\bfx, \bfe)$ is defined as an intersection number of the logarithmic double ramification cycle $\logDR_g(\bfx)$ supported on a log blowup of $\Mbar_{g,n}$. However, in the absence of branch cycles, all the insertions $\psi_i$ above are pulled back from $\Mbar_{g,n}$ and so by applying the projection formula, we can replace $\logDR_g(\bfx)$ by its pushforward $\DR_g(\bfx)$ in the above intersection number.

The splitting formula for $\psi$-classes can be used to recursively calculate the numbers \eqref{eqn:psi_descendant}. However, when $k \neq 0$ the recursion naturally features a generalization of these numbers defined in equation \eqref{eqn:psi_kappa_descendant} as
\begin{equation*} 
    \mathrm{H}_g(\bfx, \bfe, f) = \int_{\Mbar_{g,n}} \DR_g(\bfx) \cdot \psi_1^{e_1} \cdots \psi_n^{e_n} \cdot \kappa_1^f\,,
\end{equation*}
where now $|e| + f = 2g-3+n$. 

\begin{proposition}
Assume that for $1 \leq s \leq n$ the component $e_s$ of $\bfe$ is positive and denote $\bfe_s = \bfe - \delta_s$. Then
\begin{align} \label{eqn:Hxef_recursion}
x_s (2g-2+n) \cdot \mathrm{H}_g(\bfx, \bfe, f) = k \cdot \mathrm{H}_g(\bfx, \bfe_s, f+1) + \sum_{\pi: \Gamma \to T_1}\rho(\pi,s) \cdot \mult(\pi) \cdot \mathrm{Cont}_{\pi, \bfe, f}
\end{align}
where the sum goes over covers $\pi: \Gamma \to T_1$ with two vertices $v_0, v_1$ of $\Gamma$ which map to the two vertices $w_0, w_1$ of $T_1$ (in that order). The contribution of this cover to the formula above is given by
\begin{equation} \label{eqn:Cont_Gamma_T1_e}
 \mathrm{Cont}_{\pi, \bfe, f} = \binom{f}{f_0, f_1} \mathrm{H}_{g(v_0)}(\bfx[v_0], \bfe_s[v_0], f_0) \cdot \mathrm{H}_{g(v_1)}(\bfx[v_1], \bfe_s[v_1], f_1)\,,
\end{equation}
where $\bfx[v_j]$, $\bfe_s[v_j]$ are the entries of the vectors $\bfx, \bfe_s$ associated to markings attached to vertex $v_j$ and  
\[
f_j = 2g(v_j)-3+n(v_j) - |\bfe_s[v_j]|\,.
\]
\end{proposition}
\begin{proof}
This immediately follows from multiplying \eqref{eqn:psi_splitting} by $\psi_1^{e_1} \cdots \psi_s^{e_s-1} \cdots \psi_n^{e_n}$ and taking the integral over $\Mbar_{g,n}$. The factor $\binom{f}{f_0, f_1}$ arises when the factor $\kappa_1^f$ in \eqref{eqn:psi_kappa_descendant} splits as $(\kappa_{1,v_0} + \kappa_{1, v_1})^f$ when restricted to the boundary stratum associated to $\Gamma$. By dimension reasons, the only term which survives is $\kappa_{1, v_0}^{f_0} \kappa_{1, v_1}^{f_1}$, which appears with the above binomial factor. Note that when either $f_0$ or $f_1$ are negative, the integral vanishes for dimension reasons, and the binomial vanishes by definition.
\end{proof}

The numbers $\mathrm{H}_g(\bfx, \bfe, f)$ are polynomial in $\bfx$ by \cite{PZ24}. Seeing the entries $x_t$ of $\bfx$ as formal variables and dividing equation \eqref{eqn:Hxef_recursion} by $x_s (2g-2+n)$, this equation determines $\mathrm{H}_g(\bfx, \bfe, f)$ in terms of the polynomials $\mathrm{H}_{g'}(\bfx', \bfe', f')$ with $|\bfe'|<|\bfe|$. Iterating this procedure, the initial data of the recursion is given by the numbers
\begin{equation}  \label{eqn:kappa_descendant}
    \mathrm{H}_g(\bfx, \mathbf{0}, 2g-3+n) = \int_{\Mbar_{g,n}} \DR_g(\bfx) \cdot\kappa_1^{2g-3+n}\,.
\end{equation}
These numbers are determined explicitly in forthcoming work \cite{Sauvaget_integral} by Sauvaget. 

\section{Piecewise polynomiality in any genus and wall crossings for genus \texorpdfstring{$0$}{0}}

\begin{theorem}\label{thm-pp}
    Let $n \geq 3$  and let  $\mathcal{X}= \{\bfx \in \mathbb{Z}^n \;|\;|\bfx| = k (2g-2+n)\}$ for some $k \in \mathbb{Z}$. Let $\bfe \in \mathbb{Z}_{\geq 0}^n$ such that $0 \leq |\bfe| \leq 2g-3+n$.

We can then view the $k$-leaky double Hurwitz descendant as a function

$$ \mathcal{X}\rightarrow \mathbb{Q}: \bfx \mapsto \mathrm{H}_g(\bfx, \bfe). $$

This function is piecewise polynomial, where the polynomials are of degree $$ n-3+4g-|\bfe|.$$

\end{theorem}

Using the Correspondence Theorem \ref{thm-corres} the function mapping $\bfx$ to the count of tropical $k$-leaky covers satisfying Psi-conditions is also piecewise polynomial of course. In fact, we will prove Theorem \ref{thm-pp} on the tropical side.

\begin{remark}\label{rem-vertexmult}
    In genus $0$, the vertex multiplicity 
    \[ \mult_V:= \int_{\overline{M}_{0,\val(V)}} \logDR_{0}(\mathbf{a_V})\cdot \prod_{i\in I_V} \psi_i^{e_i} = \int_{\overline{M}_{0,\val(V)}}  \prod_{i\in I_V} \psi_i^{e_i} = \frac{(\val(V)-3)!}{\prod_{i\in I_V}e_i!}\] 
     of a leaky cover satisfying Psi-conditions does not depend on $k$ or on the expansion factors of its adjacent edges; it equals a multinomial coefficient that only depends on its valency and the Psi-conditions.

More generally, the vertex multiplicities are given by the intersection numbers
$$
\int_{\overline{M}_{g(v),\val(v)}} \DR_{g(v)}(\mathbf{x}(v))\cdot \prod_{i\in I_v} \psi_i^{e_i}\,.
$$
By \cite{PZ24}, the cycle $\DR_{g(v)}(\mathbf{x}(v))$ is a tautological class with coefficients which are polynomials in the entries of the vector $\mathbf{x}(v)$ of degree equal to $2 g(v)$. Accordingly, the vertex multiplicities are polynomials in the entries of the vector $\mathbf{x}(v)$ of degree equal to $2 g(v)$.
     
\end{remark}

\begin{proof}[Proof of Theorem \ref{thm-pp}]
By the Correspondence Theorem \ref{thm-corres}, $$\mathrm{H}_g(\bfx, \bfe)=\mathrm{H}^{\trop}_g(\bfx, \bfe),$$ 
 and the latter is a sum over all tropical $k$-leaky covers $\pi$ of degree $\bfx$ and genus $g$ satisfying the Psi-conditions $\bfe$ and mapping to a fixed metric line graph, where each cover $\pi$ is counted with 
 multiplicity $\mult(\pi)$ equal to the product of expansion vectors, vertex multiplicities and $\frac{1}{
|\Aut(\pi)|}$.

Given the combinatorial type of an abstract tropical curve of genus $g$ with $n$ labeled ends (such that the valence of a vertex $v$ of genus $g(v)$ adjacent to the ends with labels in $I_v$ equals $\sum_{i\in I_V} e_i+3-2g(v)$), we associate the expansion factor $|x_i|$ to the end with label $i$. 
We orient the ends pointing inward if $x_i>0$ and outward if $x_i<0$.
We view this degree now as something varying with the vector $\bfx$.

Furthermore, we fix $g':=g-\sum_v g(v)$ edges whose removal produces a tree, and view their expansion factors as variables $i_1,\ldots,i_{g'}$.

We pick an arbitrary orientation for the bounded edges, for which we ask ourselves whether there exists a map to the line graph respecting this orientation.

Using the $k$-leaky condition, the expansion factor of every other edge is then uniquely determined, and it is equal to an affine-linear form in the $x_j$ and $i_k$.

A map to the line graph exists if and only if each expansion factor is positive.

If a map to the line graph exists, there is a unique way to add a metric to the graph which is compatible with the metric of the target.

Consider the space with coordinates $i_1,\ldots,i_{g'}$. Each expansion factor defines a hyperplane equation in this space, such that the expansion factor is positive if and only if we are on the right side of the hyperplane.

The sum over all leaky tropical covers can thus be viewed as a weighted sum over all integer points $(i_1,\ldots,i_{g'})$ in a bounded chamber of a hyperplane arrangement defined by an oriented labeled graph (with orientations of the ends matching the $x_i$), where each summand contributes $\mult(\pi)$ for the associated leaky cover $\pi$. 
The fact that the chamber is bounded follows from Corollary 2.13 in \cite{CJM10}.

By Definition \ref{def-trophgxe}, the multiplicity $\mult(\pi)$ with which $\pi$ contributes to the count of $k$-leaky covers satisfying Psi-conditions is a product of expansion vectors, vertex multiplicities and $\frac{1}{
|\Aut(\pi)|}$. 
The last factor is a a number, independent of the expansion factors of the ends.
The first factor is a product of affine-linear forms in the $x_j$ and the $i_k$ of degree equal to the number of bounded edges, which is $n-3+3g'-\sum_v (\val(v)-3)= n-3+3g'-|\bfe|+\sum_v 2g(v).$

By Remark \ref{rem-vertexmult}, the vertex multiplicities are polynomial of degree $g(v)$ in the expansion factors of the adjacent edges, which are themselves affine-linear forms in the $x_j$ and the $i_k$. 

Thus, viewed as polynomial in the $x_j$ and $i_k$, the multiplicity $\mult(\pi)$ is of degree $ n-3+3g'-|\bfe|+\sum_v 4g(v).$

Summing over the points $(i_1,\ldots,i_{g'})$ in the bounded chamber increases the degree by $g'$.

Thus the multiplicity $\mult(\pi)$ is a polynomial in the $x_j$ of degree 
$ n-3+4g'-|\bfe|+\sum_v 4g(v)= n-3+4g-|\bfe|.$

In total we obtain a piecewise polynomial function, where the piecewise structure arises since the topology of the hyperplane arrangement of the expansion factors in the space with coordiantes $i_1,\ldots,i_{g'}$ may vary for different choices of $x_j$.
\end{proof}

\begin{example}\label{ex-pp}
    Let $g=0$, $n=5$, $k=1$ and $\bfe=(1,0,0,0,0)$. 
We fix the inequalities $x_1,x_4>0$, $x_2,x_3,x_5<0$, $x_1+x_4+x+5-2>0$, $x_1+x_4+x_2-2>0$, $x_1+x_4+x_5-2>0$, $x_1+x_3+x_5-2>0$, $x_1+x_3+x_4-2>0$, $x_1+x_2+x_5-2>0$.

The unique bounded edge of each the 6 labeled trees with label $1$ adjacent to a $4$-valent vertex can be oriented in a unique way producing a leaky cover with positive expansion factors in the chamber defined by these inequalities, see Figure \ref{fig-pp}. In total, we obtain for all $\bfx$ satisfying the inequalities above the polynomial 
\begin{align*} \mathrm{H}_0(\bfx, (1,0,0,0,0))=& (x_1+x_2+x_3-2)+(x_1+x_4+x_5-2)+(x_1+x_3+x_5-2)\\&+(x_1+x_3+x_4-2)+(x_1+x_2+x_5-2)+(x_1+x_2+x_4-2)\\=&
6x_1+3x_2+3x_3+3x_4+3x_5-12 = 3x_1-3.\end{align*}

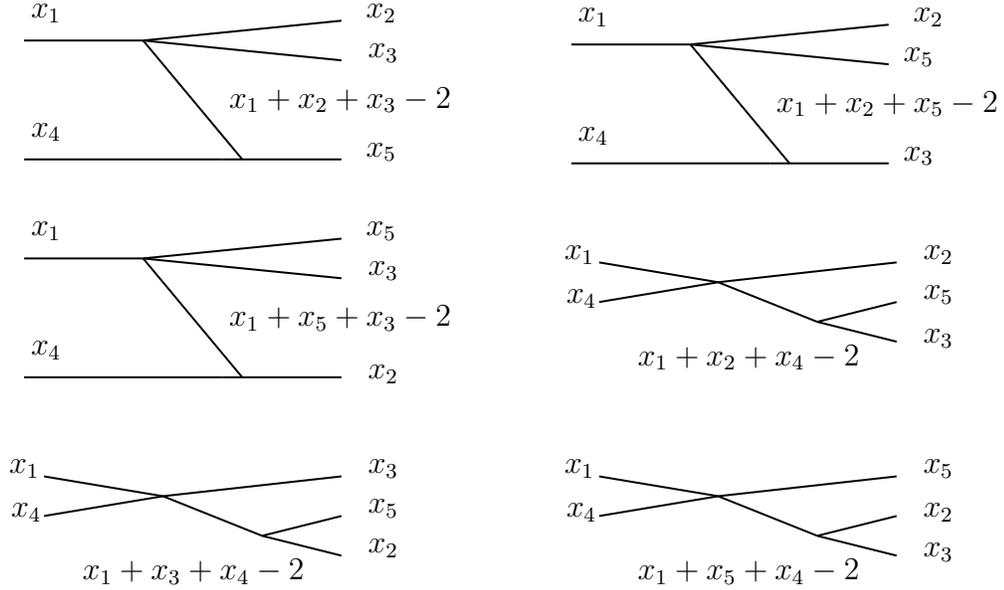
\begin{figure}[tb]
    \centering

\tikzset{every picture/.style={line width=0.75pt}} 

\begin{tikzpicture}[x=0.75pt,y=0.75pt,yscale=-1,xscale=1]

\draw    (120,200) -- (180,200) ;
\draw    (180,200) -- (230,260) ;
\draw    (180,200) -- (280,210) ;
\draw    (180,200) -- (280,190) ;
\draw    (220,260) -- (280,260) ;
\draw    (120,260) -- (220,260) ;
\draw    (396,202) -- (456,202) ;
\draw    (456,202) -- (506,262) ;
\draw    (456,202) -- (556,212) ;
\draw    (456,202) -- (556,192) ;
\draw    (496,262) -- (556,262) ;
\draw    (396,262) -- (496,262) ;
\draw    (120,310) -- (180,310) ;
\draw    (180,310) -- (230,370) ;
\draw    (180,310) -- (280,320) ;
\draw    (180,310) -- (280,300) ;
\draw    (220,370) -- (280,370) ;
\draw    (120,370) -- (220,370) ;
\draw    (410,312) -- (470,322) ;
\draw    (470,322) -- (520,342) ;
\draw    (470,322) -- (560,312) ;
\draw    (470,322) -- (410,332) ;
\draw    (520,342) -- (560,332) ;
\draw    (520,342) -- (560,352) ;
\draw    (130,420) -- (190,430) ;
\draw    (190,430) -- (240,450) ;
\draw    (190,430) -- (280,420) ;
\draw    (190,430) -- (130,440) ;
\draw    (240,450) -- (280,440) ;
\draw    (240,450) -- (280,460) ;
\draw    (410,420) -- (470,430) ;
\draw    (470,430) -- (520,450) ;
\draw    (470,430) -- (560,420) ;
\draw    (470,430) -- (410,440) ;
\draw    (520,450) -- (560,440) ;
\draw    (520,450) -- (560,460) ;

\draw (122,180) node [anchor=north west][inner sep=0.75pt]   [align=left] {$\displaystyle x_{1}$};
\draw (291,180) node [anchor=north west][inner sep=0.75pt]   [align=left] {$\displaystyle x_{2}$};
\draw (292,200) node [anchor=north west][inner sep=0.75pt]   [align=left] {$\displaystyle x_{3}$};
\draw (122,240) node [anchor=north west][inner sep=0.75pt]   [align=left] {$\displaystyle x_{4}$};
\draw (291,250) node [anchor=north west][inner sep=0.75pt]   [align=left] {$\displaystyle x_{5}$};
\draw (222,222) node [anchor=north west][inner sep=0.75pt]   [align=left] {$\displaystyle x_{1} +x_{2} +x_{3} -2$};
\draw (398,182) node [anchor=north west][inner sep=0.75pt]   [align=left] {$\displaystyle x_{1}$};
\draw (567,182) node [anchor=north west][inner sep=0.75pt]   [align=left] {$\displaystyle x_{2}$};
\draw (562,252) node [anchor=north west][inner sep=0.75pt]   [align=left] {$\displaystyle x_{3}$};
\draw (398,242) node [anchor=north west][inner sep=0.75pt]   [align=left] {$\displaystyle x_{4}$};
\draw (562,202) node [anchor=north west][inner sep=0.75pt]   [align=left] {$\displaystyle x_{5}$};
\draw (498,224) node [anchor=north west][inner sep=0.75pt]   [align=left] {$\displaystyle x_{1} +x_{2} +x_{5} -2$};
\draw (122,290) node [anchor=north west][inner sep=0.75pt]   [align=left] {$\displaystyle x_{1}$};
\draw (292,362) node [anchor=north west][inner sep=0.75pt]   [align=left] {$\displaystyle x_{2}$};
\draw (292,310) node [anchor=north west][inner sep=0.75pt]   [align=left] {$\displaystyle x_{3}$};
\draw (122,350) node [anchor=north west][inner sep=0.75pt]   [align=left] {$\displaystyle x_{4}$};
\draw (291,290) node [anchor=north west][inner sep=0.75pt]   [align=left] {$\displaystyle x_{5}$};
\draw (222,332) node [anchor=north west][inner sep=0.75pt]   [align=left] {$\displaystyle x_{1} +x_{5} +x_{3} -2$};
\draw (391,302) node [anchor=north west][inner sep=0.75pt]   [align=left] {$\displaystyle x_{1}$};
\draw (572,302) node [anchor=north west][inner sep=0.75pt]   [align=left] {$\displaystyle x_{2}$};
\draw (572,344) node [anchor=north west][inner sep=0.75pt]   [align=left] {$\displaystyle x_{3}$};
\draw (392,324) node [anchor=north west][inner sep=0.75pt]   [align=left] {$\displaystyle x_{4}$};
\draw (572,322) node [anchor=north west][inner sep=0.75pt]   [align=left] {$\displaystyle x_{5}$};
\draw (428,352) node [anchor=north west][inner sep=0.75pt]   [align=left] {$\displaystyle x_{1} +x_{2} +x_{4} -2$};
\draw (111,410) node [anchor=north west][inner sep=0.75pt]   [align=left] {$\displaystyle x_{1}$};
\draw (292,450) node [anchor=north west][inner sep=0.75pt]   [align=left] {$\displaystyle x_{2}$};
\draw (292,410) node [anchor=north west][inner sep=0.75pt]   [align=left] {$\displaystyle x_{3}$};
\draw (112,432) node [anchor=north west][inner sep=0.75pt]   [align=left] {$\displaystyle x_{4}$};
\draw (292,430) node [anchor=north west][inner sep=0.75pt]   [align=left] {$\displaystyle x_{5}$};
\draw (148,460) node [anchor=north west][inner sep=0.75pt]   [align=left] {$\displaystyle x_{1} +x_{3} +x_{4} -2$};
\draw (391,410) node [anchor=north west][inner sep=0.75pt]   [align=left] {$\displaystyle x_{1}$};
\draw (572,432) node [anchor=north west][inner sep=0.75pt]   [align=left] {$\displaystyle x_{2}$};
\draw (572,452) node [anchor=north west][inner sep=0.75pt]   [align=left] {$\displaystyle x_{3}$};
\draw (392,432) node [anchor=north west][inner sep=0.75pt]   [align=left] {$\displaystyle x_{4}$};
\draw (572,410) node [anchor=north west][inner sep=0.75pt]   [align=left] {$\displaystyle x_{5}$};
\draw (428,460) node [anchor=north west][inner sep=0.75pt]   [align=left] {$\displaystyle x_{1} +x_{5} +x_{4} -2$};

\end{tikzpicture}

     \caption{6 rational $1$-leaky covers of degree $\bfx$ satisfying the inequalities in Example \ref{ex-pp} yield a nonzero contribution to the count $ \mathrm{H}_0(\bfx, (1,0,0,0,0))$. For each, its multiplicity equals the expansion factor of its unique bounded edge.}
    \label{fig-pp}
\end{figure}

\end{example}

\begin{lemma}
  Let $g=0$, $n \geq 3$  and let  $\mathcal{X}= \{\bfx \in \mathbb{Z}^n \;|\;|\bfx| = k (-2+n)\}$ for some $k \in \mathbb{Z}$. Let $\bfe \in \mathbb{Z}_{\geq 0}^n$ such that $0 \leq |\bfe| \leq -3+n$.
The walls separating the areas of polynomiality of the piecewise polynomial function 
$ \mathrm{H}_g(\bfx, \bfe)$    are given by vanishing expansion factors, i.e.\ by expressions of the form
$$ \sum_{i\in I} x_i -k\cdot (\sharp I-1)=0, $$
where $I\subset \{1,\ldots,n\}$, $2\leq \sharp I \leq n-2$.
\end{lemma}
\begin{proof}
    The walls separating chambers of polynomiality are given by expressions as above, as this is the expansion factor of the edge separating the ends with labels in $I$ from the ends with labels not in $I$. Trees with an edge of this weight correspond to the polynomial on one side, on the other side the tree with the edge reversed contributes.
\end{proof}






    

\begin{proof}[Proof of Theorem \ref{thm-wc}]
    The proof follows the ideas presented in \cite{CJM10} for tropical double Hurwitz numbers, we have to include leaking and Psi-conditions.

Recall from the proof of Theorem \ref{thm-pp} that the polynomial expression equals a sum over oriented labeled trees such that the valence of a vertex $V$ adjacent to the ends with labels in $I_V$ equals $\sum_{i\in I_V} e_i+3$. Each tree contributes either $0$ or $\mult(\pi)$ for the leaky cover $\pi$ we can build from it. We build a cover by adding expansion factors to the bounded edges, satisfying the leaky condition, and a metric. 
If an edge in such a cover has expansion factor $\delta>0$, the corresponding oriented tree yields $0$ on the other side of the wall (as then $\delta<0$), however, the tree with the orientation of the edge reversed yields a nonzero contribution. Vice versa, that tree does not yield a contribution on the first side of the wall.

Trees that do not produce covers with an expansion factor being $\pm \delta$ contribute the same to both sides of the wall and thus do not contribute to the wall-crossing.

We produce a weighted bijection between ''cut-and-reglued'' covers and covers contributing to the wall-crossing.
Given a cover contributing to the wall-crossing, cut the edge with expansion factor $\delta$. We obtain two covers, one that contributes to $\mathrm{H}_0(\bfx_I\cup \{\delta\}, \bfe_I) $
 and one that contributes to
$\mathrm{H}_0(\bfx_{I^c}\cup \{-\delta\}, \bfe_{I^c})$.
Vice versa, if we have a pair of leaky covers, one contributing to $\mathrm{H}_0(\bfx_I\cup \{\delta\}, \bfe_I) $
 and one to
$\mathrm{H}_0(\bfx_{I^c}\cup \{-\delta\}, \bfe_{I^c})$, how can we reglue the ends labeled $\pm \delta$? First, we have to interlace the images of the vertices. There are $\binom{r}{r_1,r_2}$ choices for this, as $r$ is the number of vertices of the whole tree, while $r_i$ are the numbers of vertices of the two pieces.

For a fixed such choice, the orientation of the reglued edge labeled $\delta$ is determined by the images of the vertices and the compatibility with the cover of the line graph. Depending on this orientation, the reglued cover yields a nonzero contribution to precisely one of the sides of the wall. 
If it lives on side $1$, the only missing ingredient to count it with its correct multiplicity is the expansion factor of the cut and reglued edge, which is $\delta>0$. If it lives on side $2$, it appears with negative sign in the difference for the wall-crossing, and we have so far missed the expansion factor of its cut and reglued edge, which is $-\delta$. As the two minus signs cancel, we can treat both cases in the same way and thus obtain the claimed equality. 
\end{proof}

\begin{example}\label{ex-wc}
We continue Example \ref{ex-pp}. There, $n=5$, $k=1$ and $\bfe=(1,0,0,0,0)$. 
We computed the polynomial $ \mathrm{H}_0(\bfx, (1,0,0,0,0))$ for $\bfx$ satisfying the inequalities $x_1,x_4>0$, $x_2,x_3,x_5<0$, $x_1+x_4+x+5-2>0$, $x_1+x_4+x_2-2>0$, $x_1+x_2+x_3-2>0$, $x_1+x_3+x_5-2>0$, $x_1+x_3+x_4-2>0$, $x_1+x_2+x_5-2>0$. Let us now cross the wall $\delta= x_1+x_2+x_3-2=0$.
In Figure \ref{fig-pp}, all covers except the top left yield the same contribution on the other side of the wall, so their contributions cancel in the wall-crossing. Instead of the cover on the top left, the cover depicted in Figure \ref{fig-wc} arises on the other side of the wall. 

\begin{figure}[tb]
    \centering

\tikzset{every picture/.style={line width=0.75pt}} 

\begin{tikzpicture}[x=0.75pt,y=0.75pt,yscale=-1,xscale=1]

\draw    (120,200) -- (220,200) ;
\draw    (220,200) -- (170,260) ;
\draw    (220,200) -- (280,210) ;
\draw    (220,200) -- (280,190) ;
\draw    (220,260) -- (280,260) ;
\draw    (120,260) -- (220,260) ;

\draw (122,180) node [anchor=north west][inner sep=0.75pt]   [align=left] {$\displaystyle x_{1}$};
\draw (291,180) node [anchor=north west][inner sep=0.75pt]   [align=left] {$\displaystyle x_{2}$};
\draw (292,200) node [anchor=north west][inner sep=0.75pt]   [align=left] {$\displaystyle x_{3}$};
\draw (122,240) node [anchor=north west][inner sep=0.75pt]   [align=left] {$\displaystyle x_{4}$};
\draw (291,250) node [anchor=north west][inner sep=0.75pt]   [align=left] {$\displaystyle x_{5}$};
\draw (222,222) node [anchor=north west][inner sep=0.75pt]   [align=left] {$\displaystyle -x_{1} -x_{2} -x_{3} +2$};

\end{tikzpicture}

     \caption{The cover that arises when crossing the wall $\delta$ in Example \ref{ex-wc}.}
    \label{fig-wc}
\end{figure}
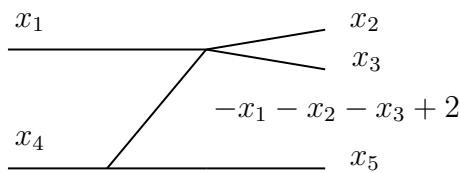

The wall-crossing equals $2(x_1+x_2+x_3-2)=\binom{2}{1,1}\delta$, as predicted by Theorem \ref{thm-wc}, as the cut $1$-leaky double Hurwitz descendants are both just one.
    
\end{example}

\section{Positivity and vanishing of leaky Hurwitz descendants in genus \texorpdfstring{$0$}{0}}



In this section we give a characterization of when genus zero leaky double Hurwitz descendant invariants vanish. By Remark \ref{rem-vertexmult}, the multiplicity with which a leaky cover contributes is always positive. A leaky double Hurwitz descendant in genus $0$ is thus positive if and only if we can construct a single leaky cover which contributes to the count.

If the genus is positive, it is possible to have leaky tropical covers which contribute with negative multiplicity, as well as some which contribute with positive multiplicity to a leaky double Hurwitz descendant, see Example \ref{ex-genus1}. The question whether a leaky double descendant is positive, negative or even $0$ is therefore hard to answer in general. 


\begin{example}\label{ex-genus1}
Consider the leaky Hurwitz number $H_1(d,-(d-2k))$ for some $d>k+1>0$. Figure \ref{fig:ex-genus1} shows two types of leaky tropical covers which contribute to the count. The upper cover has a vertex of genus $1$ with an incoming edge of weight $k$ and no outgoing edge. The vertex multiplicity of this genus $1$ vertex equals $-\frac{1}{24}$. The upper cover thus contributes $-\frac{k}{24}$.
The lower cover has only trivial vertex multiplicities. The weight $i$ of the edge in the cycle can vary from $1$ to $d-k-1$. Exchanging the two edges of the cycle yields an automorphism, so we have to divide by $\frac{1}{2}$. Altogether, the lower picture accounts for
$$\frac{1}{2}\cdot \sum_{i=1}^{d-k-1} i\cdot (d-k-i)=\frac{1}{12}\cdot (d-k)\cdot (d-k-1)\cdot (d-k+1).$$

Consequently, while the $k$-leaky number is positive for large $d$, there are regions of the parameter space where it becomes negative (e.g. $d=k+2$, $k > 12$, where the number is $\frac{1}{2}-\frac{k}{24}$).
\begin{figure}
    \centering

\tikzset{every picture/.style={line width=0.75pt}} 

\begin{tikzpicture}[x=0.75pt,y=0.75pt,yscale=-1,xscale=1]

\draw    (290,410) -- (330,400) ;
\draw  [fill={rgb, 255:red, 0; green, 0; blue, 0 }  ,fill opacity=1 ] (328.7,400) .. controls (328.7,399.28) and (329.28,398.7) .. (330,398.7) .. controls (330.72,398.7) and (331.3,399.28) .. (331.3,400) .. controls (331.3,400.72) and (330.72,401.3) .. (330,401.3) .. controls (329.28,401.3) and (328.7,400.72) .. (328.7,400) -- cycle ;
\draw    (230,410) -- (290,410) ;
\draw    (290,410) -- (380,420) ;
\draw    (230,460) -- (290,460) ;
\draw    (290,460) .. controls (300.14,451.32) and (321,451.61) .. (330,460) ;
\draw    (290,460) .. controls (299.86,469.32) and (321.86,468.75) .. (330,460) ;
\draw    (330,460) -- (380,460) ;

\draw (331.57,384.9) node [anchor=north west][inner sep=0.75pt]  [font=\tiny] [align=left] {$\displaystyle g=1$};
\draw (247.57,446.47) node [anchor=north west][inner sep=0.75pt]  [font=\scriptsize] [align=left] {$\displaystyle d$};
\draw (354.71,446.47) node [anchor=north west][inner sep=0.75pt]  [font=\scriptsize] [align=left] {$\displaystyle d-2k$};
\draw (362.71,404.18) node [anchor=north west][inner sep=0.75pt]  [font=\scriptsize] [align=left] {$\displaystyle d-2k$};
\draw (244.71,394.75) node [anchor=north west][inner sep=0.75pt]  [font=\scriptsize] [align=left] {$\displaystyle d$};
\draw (307.29,389.9) node [anchor=north west][inner sep=0.75pt]  [font=\scriptsize] [align=left] {$\displaystyle k$};
\draw (304.71,439.9) node [anchor=north west][inner sep=0.75pt]  [font=\scriptsize] [align=left] {$\displaystyle i$};
\draw (291.29,469.32) node [anchor=north west][inner sep=0.75pt]  [font=\scriptsize] [align=left] {$\displaystyle d-k-i$};

\end{tikzpicture}

    \caption{Leaky tropical covers which contribute negatively resp.\ positively to $H_1(d,-(d-2k))$.}
    \label{fig:ex-genus1}
\end{figure}
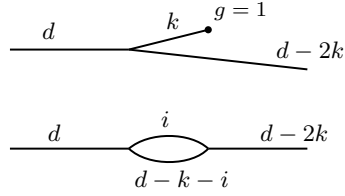
    
\end{example}

From now on we restrict to the case $g=0$, where leaky tropical covers have nonnegative multiplicity.

\begin{remark}\label{rem-k=0}
   If $k=0$, $H_{0}(\mathbf{x}, \mathbf{e})>0$ unless $\bfx=(0, \ldots, 0)$ and $n> |\mathbf{e}|+3$: for any degree $\bfx$ which is not zero, the existence of a tropical cover follows from Proposition \ref{prop-caterpillaregg}, since $\bfx$ must have negative entries. If $\bfx$ is zero, we must impose Psi-conditions that force any cover to consist of only one vertex, adjacent to all ends. Such a cover contributes positively. If we have less Psi-conditions, any cover needs to have at least one bounded edge, which must be of weight $0$ by the balancing condition, leading to a contradiction. Thus, there is no cover of degree $0$ and with $n> |\mathbf{e}|+3$.
\end{remark}

\begin{remark}\label{rem-k>0}
    In Theorem \ref{thm-positivity}, we consider the case $k\neq 0$. In the following, we assume without restriction that $k>0$. That is possible, since we can ''turn around'' any tropical $k$-leaky cover of degree $\bfx$, thus producing a tropical $(-k)$-leaky cover of degree $-\bfx$. 
\end{remark}




With the following lemma, we can deduce positivity of $k$-leaky double Hurwitz descendants from the positivity of $k$-leaky double Hurwitz numbers.

\begin{lemma}\label{lem-poswithdescendants}
If the $k$-leaky double Hurwitz number $H_0(\bfx)>0$ for some $\bfx$, then also the $k$-leaky double Hurwitz descendant $H_0(\bfx,\bfe)>0$ for any $\bfe$.    
\end{lemma}
\begin{proof}
    By definition and because of Remark \ref{rem-vertexmult}, in genus $0$, any leaky tropical cover contributes with positive multiplicity. Since $H_0(\bfx)>0$, there exists a tropical leaky cover of genus $0$ and degree $\bfx$. It has $3$-valent vertices. We can temporarily forget the order of the images of the vertices. If $\bfe\neq 0$, we can shrink bounded edges in such a way that we produce the valencies which are required by $\bfe$. We can then order the remaining vertices again in an arbitrary way (compatible with the images of the edges). The cover we produce in this way then contributes positively to $H_0(\bfx,\bfe)$ and we conclude $H_0(\bfx,\bfe)>0$.   
\end{proof}

To study the positivity of $k$-leaky double Hurwitz numbers, we first assume that $\bfx$ contains at least one entry strictly smaller than $k/2$.

\begin{proposition} \label{prop-caterpillaregg} 
    Let $k\geq 0$, $|\bfx|=k(n-2)$ and assume $\bfx$ has at least one entry $x_i<k/2$. Then there exists a caterpillar $k$-leaky cover of genus $0$ and degree $\bfx$ which contributes positively to the count $H_0(\bfx)$, see Figure \ref{fig:caterpillarandeggs}.
     As a consequence, $H_0(\bfx)>0$.
\end{proposition}
    
\begin{proof}
The case $n=3$ is trivial, since there the trivial graph with three legs is always a leaky tropical cover with multiplicity $1$ and so $H_0(\bfx)=1 > 0$. Assume that $n \geq 4$ and without loss of generality we order the markings such that $x_1 \geq x_2 \geq \ldots \geq x_n$ with $x_n<k/2$ by assumption. Then we have
\begin{equation} \label{eqn:inequality_claim}
    \textbf{Claim}: x_1 + x_2 > k \,.
\end{equation}
Indeed, assume on the contrary that $x_1 + x_2 \leq k$, then we have that $x_2 \leq k/2$ since otherwise $x_1 + x_2 \geq x_2 + x_2 > k$. By the ordering above, we then also have $x_3, \ldots, x_{n-1} \leq k/2$. Using these inequalities, we would have
\begin{align*}
    & \underbrace{x_1 + x_2}_{\leq k} + \underbrace{x_3 + \ldots + x_{n-1}}_{\leq (n-3)\cdot k/2} + x_n = k(n-2)\\
    \implies & x_n + (n-1)\cdot k/2 \geq k(n-2)\\
    \implies & x_n \geq k(n-2-\frac{n-1}{2})=k \cdot \frac{n-3}{2}
\end{align*}
Since $n \geq 4$, this gives a contradiction to the assumption $x_n<k/2$.

To conclude $H_0(\bfx)>0$ we construct a cover as follows: let $v_0$ be a vertex with markings $1,2$ and an outgoing edge, whose expansion factor is $\widetilde x = x_1 + x_2 - k > 0$ by \eqref{eqn:inequality_claim}. Then by induction, the Hurwitz number $H_0(\widetilde x, x_3, \ldots, x_n)$ is positive (since still $x_n<k/2$), and so there exists an associated leaky cover. Gluing the vertex $v_0$ in results in a cover for $H_0(\bfx)$ contributing positively, which is necessarily of the form illustrated in Figure \ref{fig:caterpillarandeggs}.
\end{proof}

     
    

\begin{figure}[tb]
    \centering

\tikzset{every picture/.style={line width=0.75pt}} 

\tikzset{every picture/.style={line width=0.75pt}} 

\begin{tikzpicture}[x=0.75pt,y=0.75pt,yscale=-1,xscale=1]

\draw    (230,330) -- (270,320) ;
\draw    (270,320) -- (330,320) ;
\draw    (220,310) -- (270,320) ;
\draw    (180,320) -- (220,310) ;
\draw    (180,300) -- (220,310) ;
\draw    (140,310) -- (180,300) ;
\draw    (140,280) -- (180,300) ;
\draw    (370,330) -- (330,320) ;
\draw    (460,310) -- (420,300) ;
\draw    (420,320) -- (380,310) ;
\draw    (330,320) -- (380,310) ;
\draw    (380,310) -- (420,300) ;
\draw    (420,300) -- (460,280) ;

\end{tikzpicture}

    \caption{A caterpillar cover: the leftmost vertices of the cover merge an end until all nonnegative ends are merged in, and the last vertices split off the negative ends.}
    \label{fig:caterpillarandeggs}
\end{figure}
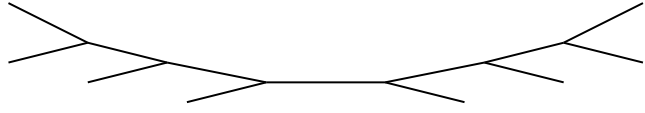

\begin{lemma} \label{Lem:classification_vanishing_e0}
     Let $k>0$ be any number and  $|\bfx|=k(n-2)$. Then the $k$-leaky double Hurwitz number $H_0(\bfx)$ is non-negative, and it is strictly positive if and only if $\bfx$ contains an entry which is not a positive multiple of $\frac{k}{2}$.
\end{lemma}
\begin{proof}
We prove the theorem by induction on $n \geq 3$. For $n=3$ the leaky Hurwitz number is always $1$, as seen above. On the other hand, the condition $x_1 + x_2 + x_3=k$ can never be satisfied if $x_i = m_i \cdot k/2$ with $m_i>0$, so the corresponding clause of the theorem is empty, and thus the leaky Hurwitz number is predicted to always be positive, which we verified.

For $n=4$, we know by Proposition \ref{prop-caterpillaregg} that $H_0(\bfx)$ is positive unless all $x_i$ are at least $k/2$. Since they sum to $2k$, the only remaining possibilty is $x_1 = \ldots = x_4 = k/2$ with $k$ even, in which case the $k$-leaky Hurwitz number vanishes, since all associated tropical covers have weight $k/2 + k/2 - k = 0$ on the bounded edge.

For $n>4$, we distinguish two cases: if $k$ is even and all $x_i$ are of the form $x_i=m_i \cdot k/2$ with $m_i > 0$, take any leaf vertex $v_0$ of a corresponding leaky cover. If the two $x_i, x_j$ adjacent to it are equal to $k/2$, the outgoing edge has slope $0$ and thus the cover does not contribute to $H_0(\bfx)$. Otherwise, the outgoing slope is a positive multiple $\widetilde m \cdot k/2$, so the remaining graph is one contributing to a leaky Hurwitz number $H_0(\widetilde m \cdot k/2, x_1, \ldots, \widehat x_i, \ldots, \widehat x_j, \ldots, x_n)$, which vanishes by induction. Thus the remaining graph has multiplicity zero, and thus has a bounded edge of weight $0$. Hence the original graph contributed with multiplicity $0$ as well and so  $H_0(\bfx)=0$.

Conversely, assume that not all numbers are of the form $x_i = m_i \cdot k/2$ for $m_i \in \mathbb{N}_{>0}$, in which case we want to show $H_0(\bfx)>0$. If any of them satisfied $x_i<k/2$, we would conclude the positivity of $H_0(\bfx)$ from Proposition \ref{prop-caterpillaregg}. Thus we can assume $x_i \geq k/2$. 

If there is at least one entry of $\bfx$, which is not a positive multiple of $k/2$, then in fact there must be two since all the $x_i$ are positive and their sum is $k(n-2)$. Assume the number of these entries is precisely two, and say they are given by $x_1, x_2$. Then necessarily $x_n=k/2$ since $n \geq 4$. Taking $v_0$ a vertex with markings $1,n$, its outgoing expansion factor is 
$$\widetilde x = \underbrace{x_1}_{\geq k/2 \text{ and } \neq k/2} + x_n -k > k/2 + k/2 -k = 0.$$ 
By induction, the number $H_0(\widetilde x, x_2, \ldots, x_{n-1})$ is positive since $x_2$ is not a positive multiple of $k/2$, and gluing $v_0$ to any cover contributing to that number gives a cover with positive contribution to $H_0(\bfx)$.

Finally, if there are at least three entries (say $x_1, x_2, x_3$) of $\bfx$ not given by positive multiples of $k/2$, then fusing $x_1, x_2$ at a vertex $v_0$ as above still has positive outgoing expansion factor (since $\widetilde x = x_1 + x_2 - k>0$), and we conclude as above using the proven case $H_0(\widetilde x, x_3, \ldots, x_n)>0$.
\end{proof}

By Lemma \ref{lem-poswithdescendants}, we can conclude:
\begin{corollary}\label{cor-notmultk/2notzero}
    Let $k>0$ be any number, and $|\bfx|=k(n-2)$. Assume  $\bfx$ contains an entry which is not a positive multiple of $\frac{k}{2}$. Then the $k$-leaky double Hurwitz descendant $H_0(\bfx,\bfe)>0$.
\end{corollary}

\begin{proposition}\label{lem-multk/2withPsi>0}
  Let $k>0$ be even and $|\bfx|=k(n-2)$. Assume $\bfx$ contains only positive integer multiples of $\frac{k}{2}$, $x_i=m_i\cdot \frac{k}{2}$.
Let $e \in \mathbb{Z}_{\geq 0}^n$ with $0 \leq |e| \leq n-3$ and  $I\subset \{1,\ldots,n\}$ be an index set of size $r$ such that \begin{equation}\label{eq:condforpos}\sum_{i\in I} e_i\geq \sum_{i\in I}m_i-r+1.\end{equation}
 Then $H_0(\bfx,\bfe)>0$.
  \end{proposition}
{
\begin{proof}

Since all the edge weights involved in the proof of this Proposition are integral multiples of $\frac{k}{2}$, we divide this factor out from all weights. The degree condition for a  leaky cover then becomes
\begin{equation}\label{leakybal}
    \sum m_i = 2n - 4.
\end{equation}
Note that \eqref{leakybal} also expresses in this normalization the leaky balancing condition at a vertex of a leaky cover, where the $m_i$'s are the weights of the incident edges, and $n$ is the valence of the vertex.
We make the following simplifying assumptions:
\begin{enumerate}
    \item for every $i\in I$, $e_i > 0$;
    \item for every $j\not\in I$, $e_j = 0$.
\end{enumerate}

We prove that these two assumptions do not cause any loss of generality. If $i \in I$  has $e_i = 0$, \eqref{eq:condforpos}  is satisfied by the subset $I\smallsetminus \{i\}$; thus given any subset $I$ satisfying \eqref{eq:condforpos}, we may replace it with the subset of its elements with strictly positive $e_i$'s, which satisfies condition $(1)$.

Assume the vector of Psi conditions $\mathbf{e}$ is supported on the subset $I\subseteq [n]$  of size $r$, and $\sum_{i\in I} e_i\geq \sum_{i\in I}m_i-r+1$.
If there exists a $k$-leaky tropical cover of degree $\bfx$ satisfying the Psi-conditions given by $\mathbf{e}$, then  for any entry-wise larger vector of Psi conditions, one can construct a $k$-leaky tropical cover as in Lemma \ref{lem-poswithdescendants} by shrinking edges. Thus condition $(2)$ poses no restriction.

With the simplifying assumptions in place, the strategy of proof is as follows: we construct a $k$-leaky cover starting from the rightmost vertex $v_R$, to which we attach all the ends in $I$; the Psi condition determines the valency of $v_R$, which requires adding  additional $s$ half-edges at the vertex. The leaky balancing condition determines the total weight of edges and ends incident to  $v_R$. We can connect all but one of the $s$ half-edges directly to ends in such a way that the weight on the last half-edge (which is determined) is still positive. To complete the picture, we are looking for a leaky cover with a bunch of left ends and exactly one  right end to glue to the remaining half-edge of $v_R$. The existence of such a graph is guaranteed by Proposition \ref{prop-caterpillaregg}. Now for the details. 

Assume without loss of generality that $I = [1,r] \subseteq [n]$, and that $m_{r+1}\leq  m_{r+2}\leq \ldots \leq m_{n}$.

Let $s = |\mathbf{e}|-r+3$, and note that $3\leq s\leq n-r$; the first inequality holds because of condition $(1)$, the second because $|\mathbf{e}|\leq n-3$. Let $\Gamma_R$ be a graph consisting of a single vertex $v_R$ of valence $r+s$. Assign weights $m_1, \ldots, m_{r+s-1}$ to all but one of the edges of $\Gamma_R$.

By \eqref{leakybal}, the weight at the last end of $\Gamma_R$ equals \begin{equation}
 m_L :=  2(r+s)-4 -\sum_{i=1}^{r+s-1}m_i.  
\end{equation}
 \noindent {\bf Claim:} $m_L>0$.

Assuming the claim, consider $\mathbf{x'} = (m_{r+s}, \ldots, m_n, -m_L)$, a vector of length $n-r-s+2$. We observe that $\mathbf{x'}$ satisfies the degree condition to admit a $k$-leaky  cover:
 \begin{align*}
     \sum_{i = r+s}^n m_i - m_L & = \sum_{i = 1}^n m_i -2(r+s)+4\\
     &= 2n - 4 -2(r+s) +4\\
     & = 2(n-r-s+2) -4.\\
 \end{align*}
Since $-m_L$ is negative, and so in particular strictly less than one, by Proposition \ref{prop-caterpillaregg} there exists a caterpillar leaky cover of degree $\mathbf{x'}$. Attaching this cover to the end of weight $m_L$ of $\Gamma_R$ produces a leaky cover of degree $\mathbf{x}$, with a rightmost vertex $v_R$ of valence
\begin{equation}\label{eq:psigood}
    r+s = |\mathbf{e}|+3.
\end{equation}
By \eqref{eq:psigood}, this graph satisfies the Psi condition $\mathbf{e}$ and contributes positively to $H_0(\mathbf{x},\mathbf{e})$. Thus the Proposition is proved modulo proving the claim.
 
\end{proof}

\begin{proof}[Proof of Claim]
Assume $m_L \leq 0$, giving
\begin{align*}
    2(r+s)-4 & \leq \sum_{i = 1}^{r+s-1} m_i  =  \sum_{i = 1}^{r} m_i+ \sum_{i = r+1}^{r+s-1} m_i\\
    & \stackrel{\eqref{eq:condforpos}}{\leq} |\mathbf{e}| +r -1 + \sum_{i = r+1}^{r+s-1} m_i\\
    & \stackrel{\eqref{eq:psigood}}{=} r+s-3 +r -1 + \sum_{i = r+1}^{r+s-1} m_i.
\end{align*}
Simplifying, we obtain
\begin{equation}
   \sum_{i = r+1}^{r+s-1} m_i \geq s, 
\end{equation}
so one of the $s-1$ summands $m_i$ on the left-hand side above must be at least $2$. In particular, since we choose to order the weights after $r$ in increasing order,  $m_i\geq 2$ for $i\geq r+s-1$.
This gives us the following contradiction:
\begin{align*}
    2n-4 = \sum_{i = 1}^{n} m_i &= \sum_{i = 1}^{r+s-1} m_i+\sum_{i = r+s}^{n} m_i\\
    & \geq  2(r+s)-4 + 2(n-r-s+1) = 2n-2. \qedhere
\end{align*}

\end{proof}}

 \begin{figure}
     \centering

\tikzset{every picture/.style={line width=0.75pt}} 

\begin{tikzpicture}[x=0.75pt,y=0.75pt,yscale=-1,xscale=1]

\draw    (260,390) -- (300,350) ;
\draw    (245,370) -- (300,350) ;
\draw    (250,305) -- (300,350) ;
\draw    (156.86,301.04) -- (184,302.71) ;
\draw    (250,335) -- (300,350) ;
\draw    (270,400) -- (300,350) ;
\draw    (240,360) -- (300,350) ;
\draw   (180,305) .. controls (180,302.24) and (195.67,300) .. (215,300) .. controls (234.33,300) and (250,302.24) .. (250,305) .. controls (250,307.76) and (234.33,310) .. (215,310) .. controls (195.67,310) and (180,307.76) .. (180,305) -- cycle ;
\draw   (180,335) .. controls (180,332.24) and (195.67,330) .. (215,330) .. controls (234.33,330) and (250,332.24) .. (250,335) .. controls (250,337.76) and (234.33,340) .. (215,340) .. controls (195.67,340) and (180,337.76) .. (180,335) -- cycle ;
\draw    (152,305.32) -- (180,305) ;
\draw    (157.71,311.32) -- (186,307.86) ;
\draw    (154.29,330.47) -- (182.57,333.29) ;
\draw    (162.86,338.18) -- (185.14,337.9) ;

\draw (212.05,304.71) node [anchor=north west][inner sep=0.75pt]  [font=\scriptsize] [align=left] {$\displaystyle \vdots $};
\draw (226.43,378.57) node [anchor=north west][inner sep=0.75pt]  [font=\scriptsize] [align=left] {$\displaystyle I$};
\draw (132.33,299.57) node [anchor=north west][inner sep=0.75pt]  [font=\scriptsize] [align=left] {$\displaystyle J_{1}$};
\draw (256.9,362.43) node [anchor=north west][inner sep=0.75pt]  [font=\scriptsize] [align=left] {$\displaystyle \vdots $};
\draw (134.9,331.29) node [anchor=north west][inner sep=0.75pt]  [font=\scriptsize] [align=left] {$\displaystyle J_{s}$};
\draw (102.43,310.32) node [anchor=north west][inner sep=0.75pt]  [font=\scriptsize] [align=left] {$\displaystyle I^{c}$};
\draw (138.05,306.14) node [anchor=north west][inner sep=0.75pt]  [font=\scriptsize] [align=left] {$\displaystyle \vdots $};

\draw (300.05,350.14) node [anchor=north west][inner sep=0.75pt]  [font=\scriptsize] [align=left] {$v_R$};
\draw (210.9,350.43) node [anchor=north west][inner sep=0.75pt]  [font=\tiny] [align=left] {$m_{i_1} \frac{k}{2}$};

\draw (260.9,400.43) node [anchor=north west][inner sep=0.75pt]  [font=\tiny] [align=left] {$m_{i_r} \frac{k}{2}$};

\draw (260.9,300.43) node [anchor=north west][inner sep=0.75pt]  [font=\tiny] [align=left] {$w_{1} \frac{k}{2}$};

\draw (242.9,318) node [anchor=north west][inner sep=0.75pt]  [font=\tiny] [align=left] {$w_{s} \frac{k}{2}$};
\end{tikzpicture}

     \caption{The rightmost vertex of a leaky cover whose degree contains only positive multiples of $\frac{k}{2}$.}
     \label{fig:lastvertex}
 \end{figure}
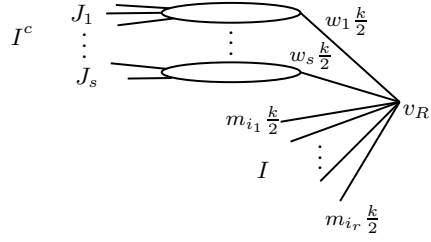

Finally, we generalize the vanishing part $H_0(\bfx)=0$ for $\bfx$ only containing positive multiples of $k/2$ from Lemma \ref{Lem:classification_vanishing_e0} to the numbers $H_0(\bfx, \bfe)$, where we see that an additional condition is required (which is always satisfied for $\bfe =0$).

    \begin{proposition}\label{prop-multk/2is0}
        Let $k>0$ be even and $|\bfx|=k(n-2)$. Assume $\bfx$ contains only positive multiples of $\frac{k}{2}$, $x_i=m_i\cdot \frac{k}{2}$.
        Let $e \in \mathbb{Z}_{\geq 0}^n$ with $0 \leq |e| \leq n-3$  and assume that for any subset $I\subset \{1,\ldots,n\}$ with $|I| = r$, we have \begin{equation}\label{eq:heresthehup}
          \sum_{i\in I}e_i<\sum_{i\in I}m_i-r+1.  
        \end{equation}
        Then $H_0(\bfx,\bfe)=0$.
    \end{proposition}

\begin{proof}
   { We show that there cannot be a leaky tropical cover of genus $0$ and degree $\bfx$ satisfying the Psi-conditions imposed by $\bfe$ in this case. Assume there is such a cover. Since $\bfx$ has only positive entries, the cover has no right ends. Thus, there cannot be an edge leaving the rightmost vertex  $v_R$ to the right.
    Denote by $I$ the subset of ends that attach directly to $v_R$, and say that an additional $s$ edges (which are not ends) attach to $v_R$, with weights $w_j\cdot \frac{k}{2}$ for $j = 1, \ldots s$, see Figure \ref{fig:lastvertex}. Denoting by $r = |I|$, the  leaky balancing condition at $v_R$ is:
    \begin{equation}
      \label{eq:leakybalvR}  
      \sum_{i\in I}m_i + \sum_{j=1}^s w_j = 2(r+s)-4.
    \end{equation}
 The cover satisfying the Psi-conditions at $v_R$ implies
 \begin{equation}
     \label{eq:psiconsisgoodatvR}
     \sum_{i\in I} e_i = r+s-3.
 \end{equation}
We may subtract \eqref{eq:psiconsisgoodatvR}
from \eqref{eq:leakybalvR}, and, using the fact that  $w_j\geq 1$ for all $j$, obtain the inequality:
\begin{equation}\label{eq:almostthere}
    \sum_{i\in I}m_i+s - \sum_{i\in I}e_i \leq r+s  -1.
\end{equation}
   It is immediate to see that \eqref{eq:almostthere} contradicts the hypothesis \eqref{eq:heresthehup}. Hence no cover can exist and $H_0(\mathbf{x}, \mathbf{e}) = 0$.
    }
\end{proof}

\begin{proof}[Proof of Theorem \ref{thm-positivity}]
By Remark \ref{rem-k>0}, we can assume $k>0$ without restriction.
{The Theorem collects together the statements of Corollary \ref{cor-notmultk/2notzero},  Proposition \ref{lem-multk/2withPsi>0} and Proposition \ref{prop-multk/2is0}.}
\end{proof}

\bibliographystyle{alpha}
\bibliography{main}

\newcommand{\etalchar}[1]{$^{#1}$}
\def\cprime{$'$}
\begin{thebibliography}{ACFW13}

\bibitem[ACFW13]{Abramovich2013Expanded-degene}
Dan Abramovich, Charles Cadman, Barbara Fantechi, and Jonathan Wise.
\newblock Expanded degenerations and pairs.
\newblock {\em Comm. Algebra}, 41(6):2346--2386, 2013.

\bibitem[ACG11]{ACG_Vol2}
Enrico Arbarello, Maurizio Cornalba, and Phillip~A. Griffiths.
\newblock {\em Geometry of algebraic curves. {V}olume {II}}, volume 268 of {\em
  Grundlehren der mathematischen Wissenschaften [Fundamental Principles of
  Mathematical Sciences]}.
\newblock Springer, Heidelberg, 2011.
\newblock With a contribution by Joseph Daniel Harris.

\bibitem[Bar19]{Barrott2019Logarithmic-Cho}
Lawrence~Jack Barrott.
\newblock Logarithmic {C}how theory.
\newblock {\em {\href{https://arxiv.org/abs/1810.03746}{arXiv:1810.03746}}},
  2019.

\bibitem[BR21]{Buryak2021Quadratic}
Alexandr Buryak and Paolo Rossi.
\newblock Quadratic double ramification integrals and the noncommutative
  {K}d{V} hierarchy.
\newblock {\em Bull. Lond. Math. Soc.}, 53(3):843--854, 2021.

\bibitem[BR23]{BuryakRossiCount}
Alexandr {Buryak} and Paolo {Rossi}.
\newblock {Counting meromorphic differentials on $\mathbb{CP}^1$}.
\newblock {\em arXiv e-prints}, page arXiv:2304.09557, April 2023.

\bibitem[BSSZ15]{BSSZ}
A.~Buryak, S.~Shadrin, L.~Spitz, and D.~Zvonkine.
\newblock Integrals of {$\psi$}-classes over double ramification cycles.
\newblock {\em Amer. J. Math.}, 137(3):699--737, 2015.

\bibitem[Bur15]{Buryak2015Double-ramifica}
Alexandr Buryak.
\newblock Double ramification cycles and integrable hierarchies.
\newblock {\em Communications in Mathematical Physics}, 336(3):1085--1107,
  2015.

\bibitem[CGH{\etalchar{+}}22]{2022arXiv221204704C}
Dawei {Chen}, Samuel {Grushevsky}, David {Holmes}, Martin {M{\"o}ller}, and
  Johannes {Schmitt}.
\newblock {A tale of two moduli spaces: logarithmic and multi-scale
  differentials}.
\newblock {\em arXiv e-prints}, page arXiv:2212.04704, December 2022.

\bibitem[CJM10]{CJM10}
Renzo Cavalieri, Paul Johnson, and Hannah Markwig.
\newblock {Tropical Hurwitz numbers}.
\newblock {\em J. Algebr. Comb.}, 32(2):241--265, 2010.
\newblock arXiv:0804.0579.

\bibitem[CJM11]{Cavalieri2011-Wall-crossings}
Renzo Cavalieri, Paul Johnson, and Hannah Markwig.
\newblock Wall crossings for double {H}urwitz numbers.
\newblock {\em Adv. Math.}, 228(4):1894--1937, 2011.

\bibitem[CM14]{cm:geomperspdhn}
Renzo Cavalieri and Steffen Marcus.
\newblock Geometric perspective on piecewise polynomiality of double {H}urwitz
  numbers.
\newblock {\em Canad. Math. Bull.}, 57(4):749--764, 2014.

\bibitem[CMR22]{CMR22}
Renzo {Cavalieri}, Hannah {Markwig}, and Dhruv {Ranganathan}.
\newblock {Pluricanonical cycles and tropical covers}.
\newblock {\em \href{https://arxiv.org/abs/2206.14034}{arXiv:2206.14034}},
  2022.

\bibitem[CMS]{Cavalieri2023-One-part-leaky}
Renzo {Cavalieri}, Hannah {Markwig}, and Johannes {Schmitt}.
\newblock One-part leaky covers.
\newblock in preparation.

\bibitem[CP23]{ChenPrado}
Dawei {Chen} and Miguel {Prado}.
\newblock {Counting differentials with fixed residues}.
\newblock {\em arXiv e-prints}, page arXiv:2307.04221, July 2023.

\bibitem[CSS21]{CSS}
Matteo Costantini, Adrien Sauvaget, and Johannes Schmitt.
\newblock Integrals of $\psi$-classes on twisted double ramification cycles and
  spaces of differentials, 2021.

\bibitem[DSvZ21]{admcycles}
Vincent Delecroix, Johannes Schmitt, and Jason van Zelm.
\newblock admcycles---a {S}age package for calculations in the tautological
  ring of the moduli space of stable curves.
\newblock {\em J. Softw. Algebra Geom.}, 11(1):89--112, 2021.

\bibitem[GJ92]{Goulden1992The-combinatorial-relation}
I.~P. Goulden and D.~M. Jackson.
\newblock The combinatorial relationship between trees, cacti and certain
  connection coefficients for the symmetric group.
\newblock {\em European J. Combin.}, 13(5):357--365, 1992.

\bibitem[GJV05]{Goulden2005Towards-the-geo}
Ian Goulden, David Jackson, and Ravi Vakil.
\newblock {Towards the geometry of double Hurwitz numbers}.
\newblock {\em Advances in Mathematics}, 198(1):43--92, 2005.

\bibitem[GT20]{GendronTahar}
Quentin {Gendron} and Guillaume {Tahar}.
\newblock {Isoresidual fibration and resonance arrangements}.
\newblock {\em arXiv e-prints}, page arXiv:2007.14872, July 2020.

\bibitem[GV05]{Graber2005Relative-virtua}
Tom Graber and Ravi Vakil.
\newblock Relative virtual localization and vanishing of tautological classes
  on moduli spaces of curves.
\newblock {\em Duke Math. J.}, 130(1):1--37, 2005.

\bibitem[HMP{\etalchar{+}}22]{Holmes_Log_DR}
D.~{Holmes}, S.~{Molcho}, R.~{Pandharipande}, A.~{Pixton}, and J.~{Schmitt}.
\newblock {Logarithmic double ramification cycles}.
\newblock {\em arXiv e-prints}, page arXiv:2207.06778, July 2022.

\bibitem[Hol19]{Holmes2017Extending-the-d}
David Holmes.
\newblock Extending the double ramification cycle by resolving the
  {Abel-Jacobi} map.
\newblock {\em J. Inst. Math. Jussieu,
  \url{https://doi.org/10.1017/S1474748019000252}}, 2019.

\bibitem[JPPZ17]{JPPZ}
Felix Janda, Rahul Pandharipande, Aaron Pixton, and Dimitri Zvonkine.
\newblock Double ramification cycles on the moduli spaces of curves.
\newblock {\em Publications math{\'e}matiques de l'IH{\'E}S}, 125(1):221--266,
  2017.

\bibitem[Li01]{Li2001Stable-morphism}
Jun Li.
\newblock Stable morphisms to singular schemes and relative stable morphisms.
\newblock {\em J. Differential Geom.}, 57(3):509--578, 2001.

\bibitem[Li02]{Li2002A-degeneration-}
Jun Li.
\newblock A degeneration formula of {GW}-invariants.
\newblock {\em Journal of Differential Geometry}, 60(2):199--293, 2002.

\bibitem[MW20]{MW17}
Steffen Marcus and Jonathan Wise.
\newblock Logarithmic compactification of the {A}bel-{J}acobi section.
\newblock {\em Proc. Lond. Math. Soc. (3)}, 121(5):1207--1250, 2020.

\bibitem[PRSS24]{LogTaut}
R.~Pandharipande, D.~Ranganathan, J.~Schmitt, and P.~Spelier.
\newblock {Logarithmic tautological rings of the moduli spaces of curves}.
\newblock {\em In preparation}, 2024.

\bibitem[PZ]{PZ24}
Aaron Pixton and Don Zagier.
\newblock On combinatorial properties of the explicit expression for double
  ramification cycles.
\newblock In preparation, see
  \href{https://websites.umich.edu/~pixton/papers/DRpoly.pdf}{here} for a
  preliminary version.

\bibitem[Sau]{Sauvaget_integral}
Adrien Sauvaget.
\newblock {Combinatorics of integrals on double ramification cycles}.
\newblock {In preparation}.

\bibitem[SSV08]{Shadrin2008-Chamber-structure}
S.~Shadrin, M.~Shapiro, and A.~Vainshtein.
\newblock Chamber behavior of double {H}urwitz numbers in genus 0.
\newblock {\em Adv. Math.}, 217(1):79--96, 2008.

\end{thebibliography}
\end{document}